\documentclass[11pt]{amsart}
\usepackage{etex}
\usepackage{geometry}
\usepackage{graphicx}
\usepackage{tikz}
\usepackage{tikz-cd}
\usepackage{amssymb}
\usepackage{url}
\usepackage{calligra}
\usepackage{stackrel}
\usepackage{wrapfig}
\usepackage{caption}
\usepackage{subfig}
\usepackage[normalem]{ulem}

\DeclareMathAlphabet{\mathscr}{T1}{calligra}{m}{n}
\theoremstyle{plain}
\newtheorem{theorem}{Theorem}

\newtheorem*{theorem*}{Theorem}
\newtheorem*{genericthm*}{\thistheoremname}
\newenvironment{namedtheorem}[1]
  {\renewcommand{\thistheoremname}{#1}%
   \begin{genericthm*}}
  {\end{genericthm*}}
\newtheorem{proposition}{Proposition}[section]
\newtheorem{corollary}[proposition]{Corollary}
\newtheorem{lemma}[proposition]{Lemma}
\theoremstyle{definition}
\newtheorem{definition}[proposition]{Definition}
\theoremstyle{remark}
\newtheorem{remark}[proposition]{Remark}
\newcommand{\thistheoremname}{}
\newtheorem*{genericrmk*}{\thistheoremname}

\newtheorem*{remark*}{Remark}
\newtheorem{example}[proposition]{Example}
\newtheorem*{example*}{Example}
\newtheorem*{exctd*}{\thistheoremname}
\newenvironment{examplectd}[1]
  {\renewcommand{\thistheoremname}{Example #1, continued}%
   \begin{exctd*}}
  {\end{exctd*}}

\newtheorem*{aside*}{Aside}
\newtheorem{question}[proposition]{Question}
\newtheorem*{question*}{Question}

\newtheorem*{conjecture*}{Conjecture}
\newtheorem*{exercise*}{Exercise}

\DeclareMathOperator{\Aut}{Aut}

\DeclareMathOperator{\rk}{rk}
\DeclareMathOperator{\Pic}{Pic}
\DeclareMathOperator{\TPic}{Pic^{\mathbb{T}}}
\DeclareMathOperator{\Picorb}{Pic}%{Pic^{orb}}
\DeclareMathOperator{\TPicorb}{Pic^{\mathbb{T}}}%{Pic^{orb}_{\mathbb{T}}}
\DeclareMathOperator{\conv}{conv}

\DeclareMathOperator{\im}{im}
\DeclareMathOperator{\coker}{coker}

\newcommand{\ds}{/\!/\,}  % double slash = orbifold quotient

\title{Picard group and quantization of toric orbifolds}

\author{Thomas~Baier} 
\email{tbaier@math.tecnico.ulisboa.pt}
\author{Jos\'e~M.~Mour\~{a}o}
\email{jmourao@math.tecnico.ulisboa.pt}
\author{Jo\~{a}o~P.~Nunes}
\email{jpnunes@math.tecnico.ulisboa.pt}
\address{Departamento de Matem\'atica\\
Instituto Superior T\'ecnico\\
Av. Rovisco Pais\\
1049-001 Lisboa\\
Portugal}

\date{\today}

\begin{document}

\begin{abstract}
In the classical theory of toric manifolds polytopes appear in two guises -- as Newton polytopes of line bundles on the complex, and as moment polytopes on the symplectic side, the link between the two being established by the prequantizability condition on the cohomology class of the symplectic form.

Here we give a combinatorial description of the orbifold Picard group for complete toric orbifolds, with the aim of detailing the relation between complex and symplectic aspects in the orbifold setting. In particular this permits to illustrate the breakdown of identification of (orbifold) line bundles by their Chern class (or moment polytope up to translations in $\mathfrak{t}^\ast$), and non-constancy of $h^0$ on representatives of the same Chern class. As an application, we discuss symplectic reduction with respect to restrictions of the action to sub-tori, and the associated Bohr--Sommerfeld conditions in mixed polarizations.
\end{abstract}

\maketitle

\tableofcontents

\section{Introduction}

In this note, we determine a combinatorial description of the orbifold Picard group and dimension of the space of global sections of orbi-line bundles on toric orbifolds over projective toric varieties. 

We need a little notation to outline our main results: consider a fixed torus $\mathbb{T} \cong \mathbb{R}^n / \mathbb{Z}^n$ with Lie algebra $\mathfrak{t}$ and fundamental group $\pi_1(\mathbb{T}) = \mathfrak{t}_{\mathbb{Z}} \subset \mathfrak{t}$; its character lattice in the dual of the Lie algebra is denoted by $\mathfrak{t}^\ast_{\mathbb{Z}} \subset \mathfrak{t}^\ast$. The well-known correspondence between fans $\Sigma$ of convex cones in the Lie algebra $\mathfrak{t}$ and toric varieties $X_\Sigma$ is reviewed in Section \ref{subsection_toric-varieties}. The additional datum needed to specify an orbifold structure (making up what is called a \emph{stacky} or \emph{weighted fan}), which together with its map to the underlying (or ``coarse'') space $X_\Sigma$ we denote by $\pi: \mathcal{X}_{\Sigma,w}\to X_\Sigma$ is a collection of weights, i.e. positive integers $w_\rho \in \mathbb{N}_+$ associated to each ray $\rho \in \Sigma(1)$ in the fan.

Line bundles on the underlying space of an orbifold permit a generalization to \emph{orbi-line bundles} by considering quotients of line bundles on orbifold charts.
Pulling back line bundles from the coarse space clearly gives a subgroup of the orbifold Picard group, that fits into a commutative diagram
  \[
  \begin{tikzcd}
    & & 0 \arrow[d] & 0\arrow[d] & \\
    & 0 \arrow[r] \arrow[d] & \Pic X_\Sigma \arrow[d] \arrow[r, "\cong"] & H^2(X_\Sigma,\mathbb{Z}) \arrow[d] \arrow[r] & 0 \\
    0 \arrow[r] & \left( \Picorb \mathcal{X}_{\Sigma,w} \right)_{tor} \arrow[r] & \Picorb \mathcal{X}_{\Sigma,w} \arrow[r, "c_1"] & c_1(\Picorb \mathcal{X}_{\Sigma,w}) \arrow[r] & 0 
  \end{tikzcd}
  \]
In these terms our main objectives are to characterize the two contributors to the orbifold Picard group beyond the underlying space, that is
  \begin{itemize}
  \item[--] identify the rational Chern classes $c \in H^2(X_\Sigma, \mathbb{Q})$ in the image of $\Picorb \mathcal{X}_{\Sigma,w}$, and
  \item[--] identify the torsion subgroup $\left( \Picorb \mathcal{X}_{\Sigma,w} \right)_{tor}$.
  \end{itemize}
  Furthermore we tie this in with the symplectic picture by turning our attention to orbi-line bundles that have powers coming from very ample bundles on the coarse space, and
  \begin{itemize}
  \item[--] identify combinatorial data refining the Newton polytope up to translation classifying very ample line bundles on $X_\Sigma$, and
    \item[--] calculate the spaces of global sections of the orbi-line bundles.
  \end{itemize}
  We finish the paper with two applications to geometric quantization of orbifolds.

Let us outline our main results in some more detail. First, we identify linearized orbi-line bundles via a combinatorial description that uses covers $\widetilde{\mathbb{T}}_\sigma$ of the torus $\mathbb{T}$ associated in particular to the maximal cones $\sigma \in \Sigma(n)$; these are defined by specifying their fundamental group $\pi_1(\widetilde{\mathbb{T}}_\sigma)$ with a sublattice $\widetilde{\mathfrak{t}}_{\sigma,\mathbb{Z}} \subset \mathfrak{t}_{\mathbb{Z}}$ defined from the combinatorial data specifying the orbifold structure.

\begin{namedtheorem}{Theorem \ref{theorem_Picorb} (pg. \pageref{theorem_Picorb})} Equivalence classes of orbi-line bundles with a lift of the $\mathbb{T}$-action on $\mathcal{X}_{\Sigma,w}$ are classified equivalently by
\begin{itemize}
 \item either compatible collections of characters $m_\sigma \in \widetilde{\mathfrak{t}}^\ast_{\sigma,\mathbb{Z}}$ of the torus covers $\widetilde{\mathbb{T}}_\sigma \to \mathbb{T}$ associated to each maximal cone $\sigma \in \Sigma(n)$,
 \item or arbitrary collections of integers $l_\rho = m_\sigma(w_\rho\nu_\rho)$ associated to the rays $\rho \in \Sigma(1)$, where $\nu_\rho$ denotes the generator of the semigroup $\rho \cap \mathfrak{t}_{\mathbb{Z}}$.
\end{itemize}
\end{namedtheorem}

In the greater generality of smooth toric Deligne--Mumford stacks, Picard groups (or rather Picard stacks) are discussed in \cite{fantechi.mann.nironi:2010}. The results there \cite[Rmks. 4.5, 5.5]{fantechi.mann.nironi:2010} could be used to obtain an explicit combinatorial description similar to ours (see Remark \ref{rmk_Fantechi} below), but this still would come at the expense of using a technically much more demanding machinery.

\begin{namedtheorem}{Theorem \ref{theorem_h0} (pg. \pageref{theorem_h0})} The dimension of the space of global sections of an orbi-line bundle $\mathcal{L}_{\{m_\sigma\}_{\sigma \in \Sigma(n)}}$ equals the number of points in the character lattice $\mathfrak{t}^\ast_{\mathbb{Z}}$ that lie in the associated Newton polytope $P_{ \{m_\sigma \}_{\sigma \in \Sigma(n)} } =  \bigcap_{\sigma \in \Sigma(n)} (m_\sigma+\sigma\check{\ })$ (where $\sigma\check{\ }\subset \mathfrak{t}^\ast$ denotes the cone dual to $\sigma \subset \mathfrak{t}$)
\[
 h^0(\mathcal{L}_{\{m_\sigma\}_{\sigma \in \Sigma(n)}}) = \# P_{\{m_\sigma\}_{\sigma \in \Sigma(n)}} \cap \mathfrak{t}^\ast_{\mathbb{Z}} .
\]
\end{namedtheorem}

Lastly, we present two applications of our results: first we determine when a symplectic toric orbifold is pre-quantizable, and the data necessary beyond the Chern class of the symplectic form to determine a prequantization, then we check ``quantization commutes with reduction'' in this setting: notationally, let $\Sigma_P$ be the fan dual to a convex polytope $P \subset \mathfrak{t}^\ast$.

\begin{namedtheorem}{Theorem \ref{theorem_prequantizability} (pg. \pageref{theorem_prequantizability})} A symplectic toric orbifold $(\mathcal{X}_{\Sigma_P,w},\omega_{P,w})$
admits an orbi-line bundle whose rational Chern class coincides with the class of the symplectic form,
\[
 \exists \mathcal{L} \to \mathcal{X}_{\Sigma_P,w} \textrm{ such that } c_1(\mathcal{L}) = [\omega_{P,w}] \in H^2(X_{\Sigma_P},\mathbb{R}) ,
\]
if and only if there exists a translation of $P$ in $\mathfrak{t}^\ast$ that takes the image of each torus fixed point to an element of the corresponding character lattice $m_\sigma \in \widetilde{\mathfrak{t}}^\ast_{\sigma,\mathbb{Z}}$, so that (after possibly translating $P$)
  \[
    P = \conv \{ m_\sigma \in \widetilde{\mathfrak{t}}^\ast_{\sigma,\mathbb{Z}} \} .
  \]

In this case, non-equivalent orbi-line bundles representing the Chern class of $\omega_{P,w}$ form a torsor under the dual $\mathfrak{t}_{(\Sigma_P,w),\mathbb{Z}}^\ast/\mathfrak{t}_{\mathbb{Z}}^\ast$ of the orbifold fundamental group,
and the dimension of the space of global sections of the line bundle \emph{does} depend on this choice.
\end{namedtheorem}

When restricting the action in this setting to a sub-torus $\mathbb{T}_1 \subset \mathbb{T}$ (with moment map $\mu_1 = \pi \circ \mu$), we encounter the notion of \emph{Bohr--Sommerfeld fibers}, as those symplectic reductions
  \[
    \mathcal{X}_\alpha = \mu_1^{-1}(\alpha) / \mathbb{T}_1
  \]
  with orbi-integral symplectic form $[\omega_\alpha] \in c_1(\Picorb \mathcal{X}_\alpha)$.  
\begin{namedtheorem}{Theorem \ref{thm_qr-rq} (pg. \pageref{thm_qr-rq})}
For any choice of orbi-line bundle $\mathcal{L}_{\{m_\sigma\}} \to \mathcal{X}_{P,w}$ representing the class of the symplectic form $\omega_{P,w}$, there is a unique orbi-line bundle $\mathcal{L}_\alpha$ on the Bohr--Sommerfeld fiber $\mathcal{X}_\alpha$ descending in a $\mathbb{T}$-equivariant manner from $\mu_1^{-1}(\alpha)$. If all reductions at Bohr--Sommerfeld values are orbifolds, we obtain an isomorphism
  \[
  H^0(\mathcal{X}_{P,w},\mathcal{L}_{\{m_\sigma\}}) \cong
  \bigoplus_{\stackrel{\alpha \in P_1}{Bohr-Sommerfeld}}
  H^0(\mathcal{X}_\alpha,\mathcal{L}_\alpha) .
  \]
\end{namedtheorem}
It follows in particular that while the set of Bohr--Sommerfeld fibers depends only on the rational Chern class of the symplectic form, the dimensions of the associated quantum space \emph{does} depend on the orbi-line bundle representing this class.

Summing up, the most important changes in the combinatorial classification data that occur compared with the familiar situation of toric varieties are thus well illustrated already in a two-dimensional setting, as exemplified in Figure \ref{fig_introduction}.

\begin{figure}[!htb]
 \centering
\subfloat[LoF 1a][ \\ $\Sigma$ fan in $\mathfrak{t}$ \\ $\updownarrow$ \\ $X_\Sigma$ toric variety]{
\label{fig_1a} 
\begin{tikzpicture}
\node[anchor=center,inner sep=0] (image1a) at (0,0) {\includegraphics[width=0.24\textwidth]{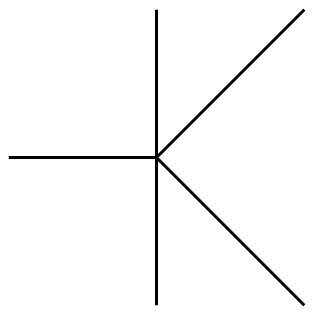}};
\begin{scope}[x={(image1a.east)-(image1a.west)},y={(image1a.north)-(image1a.south)}]
\node[label=right:{\small $\sigma_1$}] at (0.4,0){};
\node[label=right:{\small $\sigma_2$}] at (0,0.6){};
\node[label=center:{\small $\sigma_3$}] at (-0.4,0.4){};
\node[label=center:{\small $\sigma_4$}] at (-0.4,-0.4){};
\node[label=right:{\small $\sigma_5$}] at (0,-0.6){};
\end{scope}
\end{tikzpicture}}%
\subfloat[LoF 1b][ \\ $P \subset \mathfrak{t}^\ast$, $\Sigma$ refines $\Sigma_P$ \\
 $P = \conv\{m_{\sigma_i}\in\mathfrak{t}^\ast_{\mathbb{Z}}\}$ \\
 $\updownarrow$ \\
 $L_P\to X_\Sigma$ linearized line bundle]{
\label{fig_1b}
\begin{tikzpicture}
\node[anchor=center,inner sep=0] (image1b) at (0,0) {\includegraphics[width=0.24\textwidth]{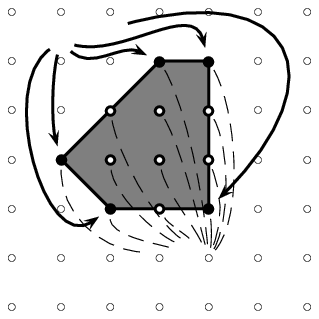}};
\begin{scope}[x={(image1b.east)-(image1b.west)},y={(image1b.north)-(image1b.south)}]
%\node[rectangle,fill,fill opacity=1,white,label=center:{\footnotesize $m_{\sigma_i} \in \mathfrak{t}^\ast_{\mathbb{Z}}$}] at (-0.6,0.85){};
\node[] at (-0.6,0.85){\footnotesize $m_{\sigma_i} \in \mathfrak{t}^\ast_{\mathbb{Z}}$};
\node[] at (0.25,-0.7){\footnotesize $h^0(L_P)=12$};
\end{scope}
\end{tikzpicture}}%
\subfloat[LoF 1c][ \\ $\Sigma$ fan in $\mathfrak{t}$ \\
 $w:\Sigma(1)\to\mathbb{N}_+$ \\
 $\updownarrow$ \\
 $\mathcal{X}_{\Sigma,w}$ toric orbifold]{
\label{fig_1c}
\begin{tikzpicture}
\node[anchor=center,inner sep=0] (image1c) at (0,0) {\includegraphics[width=0.24\textwidth]{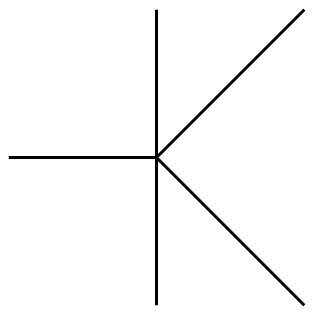}};
\begin{scope}[x={(image1c.east)-(image1c.west)},y={(image1c.north)-(image1c.south)}]
\node[label=right:{\small $\sigma_1$}] at (0.4,0){};
\node[label=right:{\small $\sigma_2$}] at (0,0.6){};
\node[label=center:{\small $\sigma_3$}] at (-0.4,0.4){};
\node[label=center:{\small $\sigma_4$}] at (-0.4,-0.4){};
\node[label=right:{\small $\sigma_5$}] at (0,-0.6){};
\node[rectangle,fill=white,fill opacity=1] at (0.3,0.3){\footnotesize $w=1$};
\node[rectangle,fill=white,fill opacity=1] at (0,0.8){\footnotesize $w=2$};
\node[rectangle,fill=white,fill opacity=1] at (-0.5,0){\footnotesize $w=2$};
\node[rectangle,fill=white,fill opacity=1] at (0,-0.8){\footnotesize $w=1$};
\node[rectangle,fill=white,fill opacity=1] at (0.3,-0.3){\footnotesize $w=1$};
\end{scope}
\end{tikzpicture}}%
\subfloat[LoF 1d][ \\ $P \subset \mathfrak{t}^\ast$, $\Sigma$ refines $\Sigma_P$ \\
 $P = \conv\{m_{\sigma_i}\in\widetilde{\mathfrak{t}}^\ast_{\sigma_i,\mathbb{Z}}\}$ \\
 $\updownarrow$ \\
 $\mathcal{L}_P\to \mathcal{X}_{\Sigma,w}$ \\
 linearized orbi-line bundle]{
\label{fig_1d}  
\begin{tikzpicture}
\node[anchor=center,inner sep=0] (image1d) at (0,0) {\includegraphics[width=0.24\textwidth]{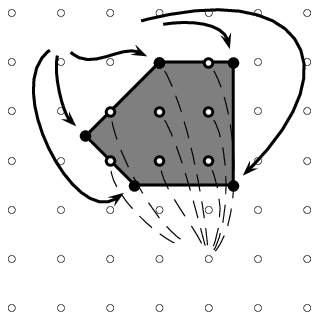}};
\begin{scope}[x={(image1d.east)-(image1d.west)},y={(image1d.north)-(image1d.south)}]
\node[] at (-0.6,0.85){\footnotesize $m_{\sigma_i} \in \widetilde{\mathfrak{t}}^\ast_{\sigma_i,\mathbb{Z}}$};
\node[rectangle,fill=white,fill opacity=1] at (0.25,-0.7){\footnotesize $h^0(\mathcal{L}_P)=8$};
\end{scope}
\end{tikzpicture}}%
\caption{Comparing line bundles on toric varieties and orbifolds}
\label{fig_introduction}   
\end{figure}

\begin{remark*}
  The main difference of this work to the substantial literature dealing with toric orbifolds and stacks and their quantization (see eg. \cite{borisov.chen.smith:2005,fantechi.mann.nironi:2010,lerman.malkin:2012,geraschenko.satriano:2015} and references therein) -- besides a ``technological downgrade'' in that we are not using the heavier machinery of stacks -- is the tightening of the relations between the description of the appropriate Picard group and the symplectic picture associated to the moment polytope. The natural and straightforward generalization of both determination of the space of sections and descent of line bundles to symplectic reductions in Theorems \ref{theorem_h0} and \ref{thm_qr-rq} illustrate this point.

  The topic of this note has also been dealt with from the stacky point of view in the pre-print \cite{sakai:2014}. It should be noted that our results on uniqueness of pre-quantization and dimension of the space of global sections differ from statements found there, cf. Remark \ref{rmk_Sakai}.
\end{remark*}

In a subsequent paper we will apply the present results in a generalization of the techniques \cite{baier.florentino.mourao.nunes:2011} to relate metric degenerations with distributional quantization in mixed polarizations.

\section{Generalities concerning toric varieties and orbifolds}

We need to fix some notation for the set-up we are concerned with.

\subsection{Toric varieties}\label{subsection_toric-varieties}

Though the symplectic variant makes an appearance in the last section, we are mainly concerned with toric varieties over the complex numbers in either the analytic or algebraic category. This is completely standard material, which we summarize here as briefly as possible.

For our purpose, it is convenient to \emph{fix} a torus $\mathbb{T} \cong \mathbb{R}^n / \mathbb{Z}^n$ and denote its Lie algebra by $\mathfrak{t}$, its dual by $\mathfrak{t}^\ast$, and the (dual) lattices of $U(1)$-subgroups and characters by $\mathfrak{t}_{\mathbb{Z}}$ and $\mathfrak{t}^\ast_{\mathbb{Z}}$, respectively; $\mathfrak{t}_{\mathbb{Z}}$ is also canonically identified with the fundamental group $\pi_1\mathbb{T}$. If we consider an element of the character lattice $m\in\mathfrak{t}_{\mathbb{Z}}^\ast$ as an actual character, we denote it by $\chi_m:\mathbb{T}\to U(1)$. Let $\mathbb{T}_{\mathbb{C}} \cong (\mathbb{C}^\ast)^n$ be the complexification of the torus.

\begin{definition} A ($\mathbb{T}_{\mathbb{C}}$-)\emph{toric variety} is an (irreducible and reduced) variety $X$ over $\mathbb{C}$ with an effective holomorphic action of $\mathbb{T}_{\mathbb{C}}$ which has a dense orbit.
\end{definition}

As is well-known, such varieties admit a completely combinatorial classification by collections of convex cones in the Lie algebra $\mathfrak{t}$.

\begin{definition} A \emph{fan} $\Sigma$ in $\mathfrak{t}$ is a collection of convex finitely generated cones $\tau \subset \mathfrak{t}$ which is closed under the inclusion of faces, and such that any two cones intersect in a common face.

A fan is \emph{rational} if every cone in it is generated by elements in $\mathfrak{t}_{\mathbb{Z}}$. It is \emph{simplicial} if every cone is so (i.e. a cone of dimension $k$ is generated by $k$ linearly independent rays). It is \emph{complete} if the union of all cones is all of $\mathfrak{t}$.
\end{definition}

We denote the collection of cones of dimension $k$ by $\Sigma(k)$, and in particular usually employ the letters $\sigma \in \Sigma(n)$ for maximal cones and $\rho \in \Sigma(1)$ for rays; we denote the generator of $\rho \cap \mathfrak{t}_{\mathbb{Z}}$ by $\nu_\rho$. Furthermore, for any cone $\tau\subset\mathfrak{t}$ we denote its dual cone by
\[
 \tau\check{\ } := \{ x \in \mathfrak{t}^\ast : \langle x, \tau \rangle_{\mathfrak{t}^\ast \times \mathfrak{t}} \geq 0 \} \subset \mathfrak{t}^\ast .
\]
\begin{definition} The \emph{dual fan} $\Sigma_P$ of a convex polytope $P$ is given by the collection of the cones dual to the cones describing $P$ locally at each vertex,
\[
 \tau \in \Sigma_P \quad \iff \quad \exists v\in P \textrm{ s.th. } P \cong v+\tau\check{\ } \textrm{ locally near } v. 
\]
\end{definition}

The relation between fans and toric varieties is described as follows:
\begin{theorem*}[{see e.g. \cite[in particular Thm. 1.5, 1.11, Cor. 2.16]{oda:1988}}] There is a one-to-one correspondence between complete normal $\mathbb{T}_{\mathbb{C}}$-toric varieties and complete rational fans $\Sigma$ in $\mathfrak{t}$.

The variety is projective if and only if the fan is dual to a convex polytope.
\end{theorem*}

Below, it will become explicit that simplicial fans yield varieties with orbifold (also ``finite quotient'') singularities only.
As we will ``lift'' the glueing construction of a toric variety to orbifold structures below, we recall quickly the functorial correspondence
\begin{equation*}
 \textrm{cones in the fan } \tau \in \Sigma
 \leftrightarrow
 \textrm{ affine toric subvarieties } U_\tau \subset X_\Sigma ,
\end{equation*}
where $U_\tau$ is defined as the closure of the embedding of $\mathbb{T}_{\mathbb{C}}$ in affine space defined by the characters that generate the semigroup $\tau\check{\ }\cap\mathfrak{t}^\ast_{\mathbb{Z}}$,
\begin{equation*}
 U_\tau \cong \overline{\im \chi_{m_1} \times \cdots \times \chi_{m_j}} \subset \mathbb{A}^{j}_{\mathbb{C}},
 \quad \textrm{ where } \quad
 \tau\check{\ }\cap\mathfrak{t}^\ast_{\mathbb{Z}} = \langle m_1, \dots, m_j \rangle .
\end{equation*}
Thus in more categorical terms, the glueing construction exhibits the toric variety $X_\Sigma$ as an injective limit of the affine toric varieties $U_\tau$ indexed by the cones $\tau\in\Sigma$,
\[
 X_\Sigma = \lim_{\underset{\tau \in \Sigma}{\longrightarrow}} U_\tau .
\]
In Corollary \ref{corollary_inj_limit} below we will prove the analogous statement for toric orbifolds. The existence of an atlas with such special properties (consisting of a finite semi-lattice of charts which are unions of torus orbits) has similar consequences in both cases.

Another fact that follows from the description of the $U_\tau$ is the character decomposition of the structure sheaf into one-dimensional subspaces,
\[
 H^0(U_\tau,\mathcal{O}_{U_\tau}) = \bigoplus_{m \in \tau\check{\ }\cap\mathfrak{t}^\ast_{\mathbb{Z}}}  H^0(U_\tau,\mathcal{O}_{U_\tau})_{\chi_m} ,
\]
as well as the sheaf of non-vanishing sections (that necessarily restrict to monomials on the open orbit $U_0 \subset U_\tau$)
\[
  H^0(U_\tau,\mathcal{O}_{U_\tau}^\times) = \coprod_{m \in \tau^\perp \cap \mathfrak{t}^\ast_{\mathbb{Z}}}  H^0(U_\tau,\mathcal{O}_{U_\tau})_{\chi_m} \setminus \{ 0 \},
\]

\subsection{Orbifolds}\label{section_orbifolds}

Useful reviews of the basics of (effective) orbifolds and further references can be found for example in \cite{haefliger.salem:1991,godinho:2001,ross.thomas:2011,putman:2012}.

\begin{definition}[{cf. \cite{haefliger.salem:1991}, \cite[\S 2.3]{putman:2012}}] A \emph{developable} (or \emph{good} in Thurston's terminology) \emph{orbifold} is a Hausdorff topological space $U$ (called the \emph{underlying} or \emph{coarse} space) together with a representation
\[
 \widetilde{\Gamma} \circlearrowright \widetilde{U} \to \widetilde{U} / \widetilde{\Gamma} \cong U
\]
as quotient of a smooth manifold $\widetilde{U}$ by a discrete group $\widetilde{\Gamma}$ acting properly discontinuously; we denote this orbifold structure on the quotient by $\widetilde{U} \ds \widetilde{\Gamma} \to \widetilde{U}/\widetilde{\Gamma}$.

Two orbifold structures $\widetilde{\Gamma}_i \circlearrowright \widetilde{U}_i \to \widetilde{U}_i / \widetilde{\Gamma}_i \cong U$ are \emph{equivalent} if there is a common covering space $V$ of the $\widetilde{U}_i$ for which the transformation groups $\Gamma_i^V := \{ g \in \Aut V \vert \exists \widetilde{g}\in\widetilde{\Gamma}_i: \psi_i \circ g = \widetilde{g}\circ \psi_i \}$ coincide,
\begin{equation}\label{diag_equiv-developments}
 \begin{tikzcd}[row sep=small]
  & V \ar{rd}{\psi_2} \ar{ld}[swap]{\psi_1} & \\
 \widetilde{U}_1 \ar{rd} & & \widetilde{U}_2 \ar{ld} \\
 & U &
 \end{tikzcd}
 \quad \textrm{ s.th. } \quad \Gamma_1^V = \Gamma_2^V .
\end{equation}

An \emph{orbifold} is a Hausdorff topological space $X$ together with an open covering $U_i$ and a set of \emph{uniformizing charts}, i.e. developable orbifold structures $\widetilde{\Gamma}_i \circlearrowright \widetilde{U}_i \to U_i$
which are pairwise equivalent on intersections $U_i \cap U_j$.
\end{definition}

\begin{remark}
The original definition of orbifold (called \emph{V-manifolds} in \cite{satake:1956}) employs a different kind of charts called \emph{Satake charts} fitting together in a \emph{Satake atlas}: these consist of an open covering $U_i$ and finite quotients $\Gamma_i\circlearrowright V_i \to U_i \cong V_i/\Gamma_i$, where $V_i$ are subsets of $\mathbb{R}^n$ (or $\mathbb{C}^n$), and $\Gamma_i \subset \Aut V_i$ are finite subgroups (called \emph{isotropy} or \emph{stabilizer} groups), such that whenever $U_i \subset U_j$, we have injections $\Gamma_i \hookrightarrow \Gamma_j$ and $V_i \hookrightarrow V_j$ (generally non-unique) that ``glue up to the actions of the stabilizer groups''.

The equivalent definition in terms of uniformizing charts serves our purpose better, however, since such exist on the canonical torus-invariant atlas of a toric variety, cf. \S \ref{sect_developable} below.
\end{remark}

\begin{definition}[cf. {\cite[\S 2.3]{putman:2012}}]\label{dfn_orbi-line-bundle}
An (analytic/algebraic) \emph{orbi-line bundle} over a developable (quasi-projective) orbifold $\mathcal{X}$ is an equivariant (analytic/algebraic) $\widetilde{\Gamma}$-equivariant line bundle on some non-singular (quasi-projective) manifold $\widetilde{U}$ such that $\mathcal{X} \cong \widetilde{U} \ds \widetilde{\Gamma}$.

Over a general orbifold, an orbi-line bundle is given by a consistent collection of orbi-line bundles on some atlas of developable charts, $L_\alpha \in \Pic^{\widetilde{\Gamma}_\alpha} \widetilde{U}_\alpha$; to write the compatibility conditions out we make use of a simplifying fact that occurs for toric orbifolds: we suppose that the atlas $U_\alpha$ is indexed by a partially ordered set that is a semi-lattice under intersections, and that the equivalence of the restrictions of orbifold charts (\ref{diag_equiv-developments}) lifts these inclusions, i.e.
\[
 \begin{tikzcd}[row sep=small]
  & \widetilde{U}_{\alpha \cap \beta} \ar{rd}{\psi^{\alpha\cap\beta}_\beta} \ar{ld}[swap]{\psi^{\alpha\cap\beta}_\alpha} & \\
 \widetilde{U}_\alpha\left\vert_{U_{\alpha \cap \beta}}\right. \ar{rd} & & \widetilde{U}_\beta\left\vert_{U_{\alpha \cap \beta}}\right. \ar{ld} \\
 & U_{\alpha \cap \beta} = U_\alpha \cap U_\beta &
 \end{tikzcd} .
\]
The glueing is specified by $\widetilde{\Gamma}_{\alpha \cap \beta}$-equivariant bundle isomorphisms
\[
 \phi_{\alpha, \alpha\cap\beta}:
 \psi^{\alpha\cap\beta}_\alpha{}^\ast L_\alpha\left\vert_{U_{\alpha \cap \beta}}\right.
 \to
 L_{\alpha\cap\beta}
\]
which give rise to transition functions subject to the usual cocycle conditions
\[
 \widetilde{\phi}_{\alpha \beta} = \phi_{\beta, \alpha\cap\beta}^{-1} \circ \phi_{\alpha, \alpha\cap\beta}
 \quad \textrm{ s.th. }\quad
 \psi^{\alpha\cap\beta\cap\gamma}_{\beta\cap\gamma}{}^\ast \widetilde{\phi}_{\beta \gamma}
 \circ
 \psi^{\alpha\cap\beta\cap\gamma}_{\alpha\cap\beta}{}^\ast \widetilde{\phi}_{\alpha \beta}
 =
 \psi^{\alpha\cap\beta\cap\gamma}_{\alpha\cap\gamma}{}^\ast \widetilde{\phi}_{\alpha \gamma}
\]
on the developable charts on triple intersections $\widetilde{U}_{\alpha \cap \beta \cap \gamma} \to U_{\alpha \cap \beta \cap \gamma}$.
\end{definition}

\begin{definition}[{\cite[\S 1.2]{haefliger.salem:1991}}] A \emph{$G$-action} on the orbifold $\mathcal{X}\to X$ is a $G$-action on the underlying space $G \circlearrowright X$ such that there is a developable atlas  $\widetilde{\Gamma}_{\alpha} \circlearrowright \widetilde{U}_{\alpha} \to U_{\alpha}$ over an atlas of $G$-invariant charts $U_\alpha \subset X$; as before the uniformizations $\widetilde{U}_\alpha$ come with canonically defined transformation groups
\[
 \widetilde{G}_{\alpha} := \{ \widetilde{g} \in \Aut \widetilde{U}_{\alpha} \vert \exists g\in G \textrm{ s.th. } \pi \circ \widetilde{g} = g \circ \pi \} 
\]
that fit into extensions $1 \to \widetilde{\Gamma}_{\alpha} \to  \widetilde{G}_{\alpha} \to G \to 1$.
\end{definition}

It is convenient for us to spell this out in the particular case of lifting a torus action to the total space of an orbi-line bundle:

\begin{definition} A \emph{$\mathbb{T}$-linearized} orbi-line bundle over a developable quasi-projective orbifold with a $\mathbb{T}$-action $\mathbb{T} \circlearrowright \mathcal{X} \cong \widetilde{U}/\widetilde{\Gamma}$ is a $\mathbb{T}_{\widetilde{U}}$-equivariant line bundle on $\widetilde{U}$ with the other conditions as before (where as above $\mathbb{T}_{\widetilde{U}} := \{ s\in\Aut \widetilde{U} \vert \exists t\in\mathbb{T} \textrm{ s.th. } \pi \circ s = t \circ \pi \}$).

Over a developable atlas as above, the bundle isomorphisms $\phi^{\alpha \cap \beta}_\alpha$ and transition functions $\widetilde{\phi}_{\alpha \beta}$ are required to be $\mathbb{T}_{\widetilde{U}_{\alpha\cap\beta}}$-equivariant.
\end{definition}

\begin{definition}\label{dfn_Picorb} Denote by $\Picorb \mathcal{X}$ the Picard groups of equivalence classes of orbi-line bundles on $\mathcal{X}$. The \emph{first Chern class} of an orbi-line bundle $\mathcal{L}$ is defined as
\[
 c_1(\mathcal{L}) := \frac{1}{d} c_1(L) \in H^2(X,\mathbb{Q}) = \Pic X \otimes \mathbb{Q}, \textrm{ where } \pi^\ast L = \mathcal{L}^d \textrm{ and } L\in\Pic X.
\]
\end{definition}
\begin{remark*}
The rational Chern class is well-defined for any orbi-line bundle since only a finite number of strata with finite isotropy groups occur, and hence a finite tensor power of any orbi-line bundle will have trivial characters at all of them and descend to a line bundle on $X$. Furthermore, if two pull-backs give rise to powers of an orbi-line bundle, then an appropriate product of them will be trivial as orbi-line bundle and thus have a trivializing section. This however descends to the coarse space, which guarantees the two rational Chern classes for $\mathcal{L}$ to coincide.
\end{remark*}

\begin{definition}[cf. \cite{davis:2011}, also for a discussion of alternative definitions]\label{dfn_orbi_pi1} An \emph{orbifold covering map} is one that locally on developable charts $\widetilde{\Gamma} \circlearrowright \widetilde{U} \to \widetilde{U} / \widetilde{\Gamma}$ is given by intermediate quotients of equivalent orbifold structures on some cover $V\to\widetilde{U}$.

The \emph{orbifold fundamental group} $\pi_1^{orb}\mathcal{X}$ is the group of deck transformations of the universal covering orbifold.
\end{definition}

This definition of course relies on an existence theorem that we do not state -- in our setting we will see an explicit construction of the orbifold universal cover below. The two facts on orbifold fundamental groups we will rely upon below are the following (see \cite[\S\S 1.1, 5.1]{haefliger.salem:1991}):
\begin{itemize}
 \item For a developable orbifold $\mathcal{X} = \widetilde{U} \ds \widetilde{\Gamma}$, we have a short-exact sequence
\[
 1 \to \pi_1 \widetilde{U} \to \pi_1^{orb}\mathcal{X} \to \widetilde{\Gamma} \to 1 ;
\]
in particular its orbifold universal cover is the universal cover of $\widetilde{U}$.
 \item
The orbifold fundamental group is a quotient of the (usual) fundamental group of the \emph{regular part} $X^{reg}$ (i.e. the set of smooth points with trivial orbifold stabilizer group) of the coarse space.
\end{itemize}

\section{Orbifold structures on simplicial projective toric varieties}

Toric orbifolds have been discussed from two technical view-points, either as generalizations of manifolds (note that the symplectic orbifolds of \cite{lerman.tolman:1997} all admit complex structures, cf. \cite[Thm. 1.7]{lerman.tolman:1997}) or as special Deligne--Mumford stacks (in particular by \cite{borisov.chen.smith:2005,iwanari:2009,fantechi.mann.nironi:2010}). The following classification theorem provides us with the combinatorial codification of toric orbifolds.

\begin{definition} A \emph{stacky} (or \emph{weighted}) fan $(\Sigma,w)$ is a (simplicial, complete, rational) fan $\Sigma$ with a weight function $\Sigma(1) \ni \rho \mapsto w_\rho \in \mathbb{N}_+$.
\end{definition}
\begin{theorem*}[{\cite[Thm. 7.17]{fantechi.mann.nironi:2010}}] Any toric orbifold $\mathcal{X}$ (with coarse space $X_{\Sigma}$ the complete simplicial toric variety with corresponding fan $\Sigma$) is isomorphic to the smooth toric Deligne--Mumford stack obtained from some stacky fan $(\Sigma,w)$.
\end{theorem*}

In what follows we describe this toric orbifold $\mathcal{X}_{\Sigma,w}\to X_\Sigma$ by an elementary glueing construction. Already in the simplest example, the weighted projective line, it is clear that there is, in general, no Satake atlas of the orbifold structure supported on the usual atlas of affine toric subvarieties -- as all points in the open orbit have trivial stabilizer group, the Satake chart on it is the identity, so there will be no immersion that lifts the inclusion into a (ramified) vertex chart. Using orbifold charts on non-invariant domains on the other hand is cumbersome for the description of linearized orbi-line bundles. The way out is indicated by  the fact that the restriction of the orbifold structure to each affine toric subvariety is a developable orbifold.

The central enhancement of the combinatorial data in the orbifold setup is the use of weights and ray generators to define a family of torus covers $\widetilde{\mathbb{T}}_\tau \to \mathbb{T}$ indexed by the cones in $\Sigma$, such that inclusions of cones $\tau_1 \subset \tau_2$ induce surjections $\widetilde{\mathbb{T}}_{\tau_1} \to \widetilde{\mathbb{T}}_{\tau_2}$. To specify $\widetilde{\mathbb{T}}_\tau$ together with the claimed homomorphisms, it is of course sufficient to give the fundamental groups $\pi_1\widetilde{\mathbb{T}}_{\tau} \subset \pi_1\mathbb{T}$.

\begin{definition} For each maximal cone $\sigma \in \Sigma(n)$, the cover $\widetilde{\mathbb{T}}_\sigma \to \mathbb{T}$ is defined by specifying
\[
 \pi_1\widetilde{\mathbb{T}}_{\sigma} =
 \widetilde{\mathfrak{t}}_{\sigma,\mathbb{Z}} := \langle w_\rho \nu_\rho \vert \forall \rho \subset \sigma, \rho \in \Sigma(1) \rangle \subset \mathfrak{t}_{\mathbb{Z}} \cong \pi_1\mathbb{T}.
\]
For an arbitrary cone $\tau \in \Sigma$, we use the same construction for
\[
 \pi_1\widetilde{\mathbb{T}}_{\tau} =
 \widetilde{\mathfrak{t}}_{\tau,\mathbb{Z}} := \bigcap_{\overset{\sigma \in \Sigma(n)}{\sigma \supset \tau}}  \widetilde{\mathfrak{t}}_{\sigma,\mathbb{Z}} .
\]
For each $\tau$, we denote the kernel of these torus covers by $\widetilde{\Gamma}_\tau := \ker \widetilde{\mathbb{T}}_\tau \to \mathbb{T}$.
\end{definition}

It is important not to confuse the groups $\widetilde{\Gamma}_\tau$ that occur in the toric presentation of the orbifold structure with the orbifold structure groups $\Gamma_\tau$ at any point of the relative interior of the torus orbit labelled by $\tau$. Only for maximal cones $\sigma$ we have that $\widetilde{\Gamma}_\sigma$ equals the stabilizer of the corresponding torus fixed point.

\begin{remark*}
Sublattices of the torus' fundamental group associated to the cones also appear in the literature on toric stacks \cite[\S 4]{borisov.chen.smith:2005}, \cite[\S 2]{goldin.harada.johannsen.krepski:2016}; these coincide with $\widetilde{\mathfrak{t}}_{\tau,\mathbb{Z}}$ only for \emph{maximal} cones $\tau = \sigma$. For the smaller cones, they are distinct since they are more directly related to the orbifold inertia groups over the corresponding torus orbits.
\end{remark*}

\begin{example}\label{ex_weightedP1} While in one dimension there is a unique complete toric variety $\mathbb{P}^1$, in the orbifold setting we encounter a whole family $\mathbb{P}^1_{n,m}$ of weighted projective lines depending on two non-negative integers that give the order of the (cyclic) stabilizer groups at the two torus fixed points. The corresponding weighted fan $\Sigma =\{ \sigma_{\pm} = \pm \mathbb{R}_{+}, 0 \},w = \{ w_- = n, w_+ = m \}$ and the lattices derived from it for the example $n=6,m=4$ are shown in Figure \ref{fig_weightedP1}.

\begin{figure}[!htb]
 \centering
\begin{tikzpicture}
\node[anchor=center,inner sep=0] (image2) at (0,0) {\includegraphics[width=\textwidth]{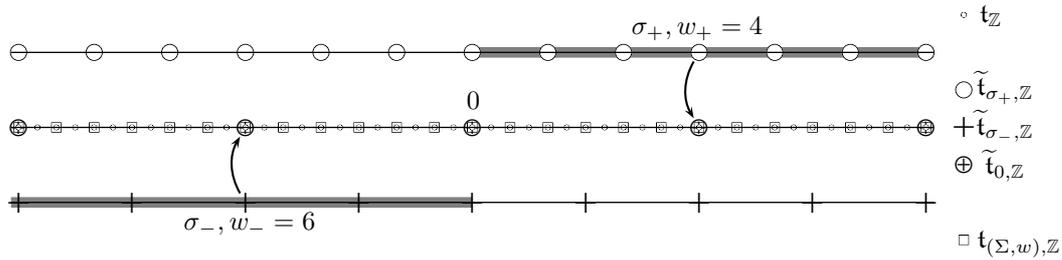}};
\begin{scope}[x={(image2.east)-(image2.west)},y={(image2.north)-(image2.south)}]
\node[] at (0.85,0){\small $\widetilde{\mathfrak{t}}_{\sigma_-,\mathbb{Z}}$};
\node[] at (0.85,0.3){\small $\widetilde{\mathfrak{t}}_{\sigma_+,\mathbb{Z}}$};
\node[] at (0.85,-0.3){\small $\widetilde{\mathfrak{t}}_{0,\mathbb{Z}}$};
\node[] at (0.825,0.85){\small $\mathfrak{t}_{\mathbb{Z}}$};
\node[] at (0.875,-0.9){\small $\mathfrak{t}_{(\Sigma,w),\mathbb{Z}}$};
\node[] at (-0.1,0.2){\small $0$};
\node[] at (0.3,0.75){\small $\sigma_+,w_+=4$};
\node[] at (-0.5,-0.75){\small $\sigma_-,w_-=6$};
\end{scope}
\end{tikzpicture}%
\caption{Weighted fan and fundamental group of covering tori of $\mathbb{P}^1_{6,4}$}\label{fig_weightedP1}
\end{figure}

In this example
\[
 \widetilde{\mathfrak{t}}_{0,\mathbb{Z}} = \widetilde{\mathfrak{t}}_{\sigma_{-},\mathbb{Z}} \cap \widetilde{\mathfrak{t}}_{\sigma_{+},\mathbb{Z}}
 \ \iff \ 
 12\mathbb{Z} = 6\mathbb{Z} \cap 4\mathbb{Z}
\]
under the natural identification $\mathfrak{t}_{\mathbb{Z}} = \mathbb{Z}$. (The lattice $\mathfrak{t}_{(\Sigma,w),\mathbb{Z}}$ is related to the orbifold fundamental group and will be introduced in general below).
Using this family of sub-lattices $\widetilde{\mathfrak{t}}_{\tau,\mathbb{Z}}$ in general, we next construct toric ``developments'' of the orbifold structure over the corresponding affine subvarieties.
\end{example}

\subsection{Affine toric orbifolds are developable}\label{sect_developable}
Each cone $\tau \in \Sigma$ determines an affine toric subvariety $U_\tau \in X_\Sigma$, so that $X_\Sigma$ is obtained by glueing these; in this paragraph we describe the orbifold structure of $\mathcal{X}_{\Sigma,w}$ by lifting this glueing construction to developable orbifold structures.

\begin{lemma}\label{lemma_developable} For every cone $\tau \in \Sigma$, considering it as defining an affine toric variety with respect to the lattices $\widetilde{\mathfrak{t}}_{\tau,\mathbb{Z}} \subset \mathfrak{t}_{\mathbb{Z}}$ defines a $\widetilde{\mathbb{T}}_{\tau} \to \mathbb{T}$-equivariant morphism of toric varieties which is a quotient by $\widetilde{\Gamma}_\tau$
\[
 \begin{tikzcd}[column sep=tiny]
 \widetilde{\mathbb{T}}_\tau \ar[draw=none]{r}[anchor=center]{\circlearrowright} \ar{d} & \widetilde{U}_\tau \ar{d}{/ \widetilde{\Gamma}_\tau} \\
  \mathbb{T}  \ar[draw=none]{r}[anchor=center]{\circlearrowright}  & U_\tau
 \end{tikzcd}
\]
with $\widetilde{U}_\tau$ non-singular.
\end{lemma}
\begin{proof} Everything is immediate from the construction of toric varieties, except non-singularity of the ``ramified cover'' $\widetilde{U}_\tau$: for this, we need to show that we can supplement the vectors $\{ w_\rho \nu_\rho \}_{\rho \subset \tau}$ to obtain a $\mathbb{Z}$-basis of the lattice $\widetilde{\mathfrak{t}}_{\tau,\mathbb{Z}}$; in particular, for $\tau = \sigma \in \Sigma(n)$ of maximal dimension there is nothing to show as the cone is simplicial by hypothesis. We now proceed by induction: suppose the assertion is proved for all cones of dimension at least $k$, and consider a pair of cones $\tau' \subset \tau$ of dimensions $k-1$ and $k$; the snake lemma
\[
 \begin{tikzcd}
 & 0 \ar{r} \ar{d} & 0 \ar{r} \ar{d} & \ker \ar{d} \ar[out=-22.5, in=157.5, looseness=3.2,overlay]{dddll} & \\
 0 \ar{r} &
 \mathbb{Z}w_{\rho_1}\nu_{\rho_1}\oplus\dots\oplus\mathbb{Z}w_{\rho_{k-1}}\nu_{\rho_{k-1}} \ar{r} \ar{d} &
 \widetilde{\mathfrak{t}}_{\tau',\mathbb{Z}} \ar{r} \ar{d} &
 Q_{k-1} \ar{r} \ar{d} & 0 \\
 0 \ar{r} & 
 \mathbb{Z}w_{\rho_1}\nu_{\rho_1}\oplus\dots\oplus\mathbb{Z}w_{\rho_k}\nu_{\rho_k} \ar{r} \ar{d} &
 \widetilde{\mathfrak{t}}_{\tau,\mathbb{Z}} \ar{r} \ar{d} &
 Q_k \ar{r} \ar{d} & 0 \\
 & \mathbb{Z} \ar{r} & \coker_1 \ar{r} & \coker_2 \ar{r} & 0
 \end{tikzcd}
\]
provides an injection $\ker \subset \mathbb{Z}$, which is therefore free of rank $1$; $\coker_1$ is finite, hence so is $\coker_2$, so $\rk Q_{k-1} = \rk Q_k +1$ and $Q_{k-1}$ has to be torsion-free if $Q_k$ is, as the generator of $\ker$ needs to map to a non-torsion element.
\end{proof}

\begin{corollary}\label{corollary_inj_limit} The collection of developable orbifold structures on $U_\tau$ glue to give a toric orbifold structure $\mathcal{X}_{\Sigma,w}$ on $X_\Sigma$,
or in more categorical terms the toric orbifold $\mathcal{X}_{\Sigma,w}$ is an injective limit of developable orbifolds
\[
 \mathcal{X}_{\Sigma,w} = \lim_{\underset{\tau \in \Sigma}{\longrightarrow}} \widetilde{U}_\tau \ds \widetilde{\Gamma}_\tau .
\]
Furthermore, since toric orbifold structures on toric varieties are uniquely determined by the orders of the cyclic stabilizer groups, the injective limit $\mathcal{X}_{\Sigma,w}$ is isomorphic to the orbifold structure obtained from either the toric Deligne--Mumford stack, or the construction from a weighted polytope in symplectic geometry.
\end{corollary}
\begin{proof} The open cover $\{U_{\tau}\}_{\tau \in \Sigma}$ is closed under intersections, and whenever $U_{\tau_1} \subset U_{\tau_2}$ there is a commutative diagram
\[
 \begin{tikzcd}[row sep=small]
 \widetilde{\mathbb{T}}_{\tau_1}  \ar[draw=none]{d}[anchor=center]{\circlearrowright} \ar[two heads]{r} & \widetilde{\mathbb{T}}_{\tau_2}  \ar[draw=none]{d}[anchor=center]{\circlearrowright} \\
 \widetilde{U}_{\tau_1} \ar{dd}{/  \widetilde{\Gamma}_{\tau_1}} \ar{r} & \widetilde{U}_{\tau_2} \ar{dd}{/  \widetilde{\Gamma}_{\tau_2}} \\
 & & \\
 U_{\tau_1} \ar[draw=none]{r}[anchor=center]{\subset} & U_{\tau_2}
 \end{tikzcd}
\]
relating the uniformizing charts; we need to show that the map $\widetilde{U}_{\tau_1} \to \widetilde{U}_{\tau_2}$ is a local diffeomorphism.
In fact, this map is the quotient of $\widetilde{U}_{\tau_1}$ by the subgroup $K \subset \widetilde{\Gamma_{\tau_1}}$ such that $\widetilde{\Gamma}_{\tau_1} / K \cong \widetilde{\Gamma}_{\tau_2}$, so we need to verify that this group acts on $\widetilde{U}_{\tau_1}$ without fixed points.
Now, for any affine toric variety $U_\tau$, the set of elements of the torus that fix \emph{some} point is the subtorus whose Lie algebra is spanned by the cone $\tau$. From the definitions it is however clear that
\[
 \widetilde{\mathfrak{t}}_{\tau_1,\mathbb{Z}} \cap \bigoplus_{\rho \subset \tau_1} \mathbb{R} \nu_\rho =
 \bigoplus_{\rho \subset \tau_1} \mathbb{Z}w_\rho \nu_\rho =
 \widetilde{\mathfrak{t}}_{\tau_2,\mathbb{Z}} \cap \bigoplus_{\rho \subset \tau_1} \mathbb{R} \nu_\rho ,
\]
so $K \cong \widetilde{\mathfrak{t}}_{\tau_2,\mathbb{Z}} / \widetilde{\mathfrak{t}}_{\tau_1,\mathbb{Z}}$ acts without fixed points and the map $\widetilde{U}_{\tau_1} \to \widetilde{U}_{\tau_2}$ is a local diffeomorphism, as required.
\end{proof}

\subsection{The orbifold fundamental group} 
This (defined as the group of deck transformations of the orbifold universal cover, cf. Definition \ref{dfn_orbi_pi1}) is an example of one feature of toric orbifolds where already the simplest possible case, the weighted projective line, provides a quite good illustration of what to expect:

\begin{examplectd}{\ref{ex_weightedP1}} The orbifold fundamental group of $\mathbb{P}^1_{n,m}$ can be calculated using the uniformizing atlas over the affine toric pieces:
\[
 \begin{tikzcd}[row sep=small]
 \widetilde{\Gamma}_{-} \ar[draw=none]{d}[anchor=center]{\circlearrowright} & \widetilde{\Gamma}_{0} \ar[draw=none]{d}[anchor=center]{\circlearrowright} \ar[two heads]{l} \ar[two heads]{r} & \widetilde{\Gamma}_{+} \ar[draw=none]{d}[anchor=center]{\circlearrowright} \\
 \widetilde{U}_{-} \ar{dd}{n:1} & \widetilde{U}_{0} \ar{dd}{k:1} \ar{l} \ar{r} & \widetilde{U}_{+} \ar{dd}{m:1} \\
 & & & \\
 U_{-} & U_{0} \ar[draw=none]{l}[anchor=center]{\supset}  \ar[draw=none]{r}[anchor=center]{\subset} & U_{+}
 \end{tikzcd}, \textrm{ where } k := \frac{nm}{\gcd(n,m)} .
\]
The universal orbifold cover of $\mathbb{P}^1_{n,m}$ is given by ``the same developable orbifold charts on the maximal quotients which still glue'': setting $\widetilde{\Gamma}'_j < \widetilde{\Gamma}_j$ to be the subgroup of index $\gcd(n,m)$ and $V_j := \widetilde{U}_j / \widetilde{\Gamma}'_j$, these charts glue to give a $\mathbb{P}^1_{n',m'}$ with $n' := n/\gcd(n,m), m' := m/\gcd(n,m)$.
Since the indices of the subgroups by which we quotient are all the same, the quotients are canonically identified
\[
 \widetilde{\Gamma}_{-} / \widetilde{\Gamma}'_{-} =  \widetilde{\Gamma}_{0} / \widetilde{\Gamma}'_{0} =  \widetilde{\Gamma}_{+} / \widetilde{\Gamma}'_{+} 
\]
and the induced actions on $V_j$ glue to a global action. In the example depicted previously $n=6,m=4$, we have $k=12$ and the orbifold universal cover is
\[
 \mathbb{Z}/2\mathbb{Z} \circlearrowright \mathbb{P}^1_{3,2} \to \mathbb{P}^1_{6,4} ,
\]
so that in particular $\pi_1^{orb} \mathbb{P}^1_{6,4} \cong \mathbb{Z} / 2\mathbb{Z}$.
\end{examplectd}

In this example, it is well known that $\mathbb{P}^1_{n,m}$ is orbifold-simply connected for $n$ and $m$ relatively prime; this follows also from the construction in the general case that we give now.

\begin{proposition}\label{prop_orbi-pi_1} The orbifold fundamental group of $\mathcal{X}_{\Sigma,w}$ is
\[
 \pi_1^{orb} \mathcal{X}_{\Sigma,w} = \mathfrak{t}_{\mathbb{Z}} / \mathfrak{t}_{(\Sigma,w),\mathbb{Z}} ,
 \textrm\qquad { where } \qquad
 \mathfrak{t}_{(\Sigma,w),\mathbb{Z}} := \langle w_\rho \nu_\rho \vert \rho \in \Sigma(1) \rangle \subset \mathfrak{t}_{\mathbb{Z}}.
\]
\end{proposition}
\begin{remark} This result is well known in the literature, cf. \cite[Thm. 3.2]{poddar.sarkar:2010}, \cite[\S 2.2]{goldin.harada.johannsen.krepski:2016}. Since it has a short and explicit interpretation using the developable charts we reprove it here.
\end{remark}
\begin{proof}
As noted in the introduction, the orbifold fundamental group $\pi_1^{orb} \mathcal{X}_{\Sigma,w}$ is a quotient of the fundamental group of the subset of regular points $X_\Sigma^{reg} \subset X_\Sigma$.
This, in turn, is a union of torus orbits which is an open subset, so it is also a union of affine toric subvarieties $U_\tau$ for certain $\tau$'s. All of these intersect in the open orbit $U_0 \subset X_\Sigma^{reg}$, and by the Seifert--van Kampen theorem we obtain a canonical epimorphism $\mathfrak{t}_{\mathbb{Z}} = \pi_1 U_0 \to \pi_1 X_\Sigma^{reg}$. We conclude that the orbifold fundamental group is a quotient
\begin{equation}\label{eqn_epi_pi1}
 \mathfrak{t}_{\mathbb{Z}} \twoheadrightarrow \pi_1^{orb} \mathcal{X}_{\Sigma,w} .
\end{equation}
On the other hand, we know that the fundamental group of a quotient of a simply connected manifold, such as $\mathcal{X}_{\Sigma,w}\vert_{U_\sigma} = \widetilde{U}_\sigma \ds \widetilde{\Gamma}_\sigma$, is just the group of deck transformations $\widetilde{\Gamma}_\sigma = \mathfrak{t}_{\mathbb{Z}} / \widetilde{\mathfrak{t}}_{\sigma,\mathbb{Z}}$, hence (since $U_0 \subset U_\sigma$) the epimorphism (\ref{eqn_epi_pi1}) has to factor through $\widetilde{\Gamma}_\sigma$ for all maximal cones $\sigma \in \Sigma(n)$. Therefore it also must factor through $\mathfrak{t}_{\mathbb{Z}} / \mathfrak{t}_{(\Sigma,w),\mathbb{Z}}$, as any ray is contained in some maximal cone, so the orbifold fundamental group must be a quotient
\begin{equation}\label{eqn_epi_pi1_2}
 \mathfrak{t}_{\mathbb{Z}} / \mathfrak{t}_{(\Sigma,w),\mathbb{Z}} \twoheadrightarrow \pi_1^{orb} \mathcal{X}_{\Sigma,w} .
\end{equation}

It is clear on the other hand that the weighted fan $(\Sigma,w)$ defines a toric \emph{variety} $Y_{(\Sigma,w)}$ with respect to the torus $\mathbb{T}_{(\Sigma,w)}$ with fundamental group $\mathfrak{t}_{(\Sigma,w),\mathbb{Z}} = \pi_1 \mathbb{T}_{(\Sigma,w)}$; we show that it is the orbifold universal cover by ``fitting it between the developable orbifold charts'' $\widetilde{U}_\tau \to U_\tau$ of $\mathcal{X}_{\Sigma,w}$.
Since $\widetilde{\mathfrak{t}}_{\tau,\mathbb{Z}} \subset \mathfrak{t}_{(\Sigma,w),\mathbb{Z}}$ for all $\tau$, we have in fact canonical epimorphisms $\widetilde{\Gamma}_\tau = \mathfrak{t}_{\mathbb{Z}} / \widetilde{\mathfrak{t}}_{\tau,\mathbb{Z}} \to \mathfrak{t}_{\mathbb{Z}} / \mathfrak{t}_{(\Sigma,w),\mathbb{Z}}$; denote their kernels by $K_\tau$, and consider the intermediate quotients
\[
 \begin{tikzcd}[column sep=scriptsize]
 & \widetilde{\mathbb{T}}_\tau \ar[draw=none]{r}[anchor=center]{\circlearrowright} \ar{dl} & \widetilde{U}_\tau \ar{dl}[anchor=center,rectangle,fill=white,fill opacity=1]{K_\tau} \ar{dd}{\widetilde{\Gamma}_\tau} \\
 \mathbb{T}_\Sigma \ar[draw=none]{r}[anchor=center]{\circlearrowright} \ar{dr} & V_\tau \ar{dr}[anchor=center,rectangle,fill=white,fill opacity=1]{\mathfrak{t}_{\mathbb{Z}} / \mathfrak{t}_{\Sigma,\mathbb{Z}} } & \\
 & \mathbb{T} \ar[draw=none]{r}[anchor=center]{\circlearrowright} & U_\tau
 \end{tikzcd} .
\]
These are just the affine toric pieces of $Y_{(\Sigma,w)}$, so we constructed the claimed $\mathfrak{t}_{\mathbb{Z}} / \mathfrak{t}_{\Sigma,\mathbb{Z}}$-action on $Y_{(\Sigma,w)}$ that yields $X_\Sigma$ as a quotient, which finishes the proof.
\end{proof}

\begin{example}\label{ex_P2_2,2,3} We consider the orbifold $\mathbb{P}^2_{2,2,3}$ with underlying space the projective plane $\mathbb{P}^2$ and weights $2,2$ and $3$ along the toric divisors, corresponding to the fan of Figure \ref{fig_weightedP2}.

\begin{figure}[!htb]
 \centering
\subfloat[LoF 3a][ \\ Sub-lattices of maximal cones]{
\label{fig_3a} 
\begin{tikzpicture}
\node[anchor=center,inner sep=0] (image3a) at (0,0) {\includegraphics[width=0.45\textwidth]{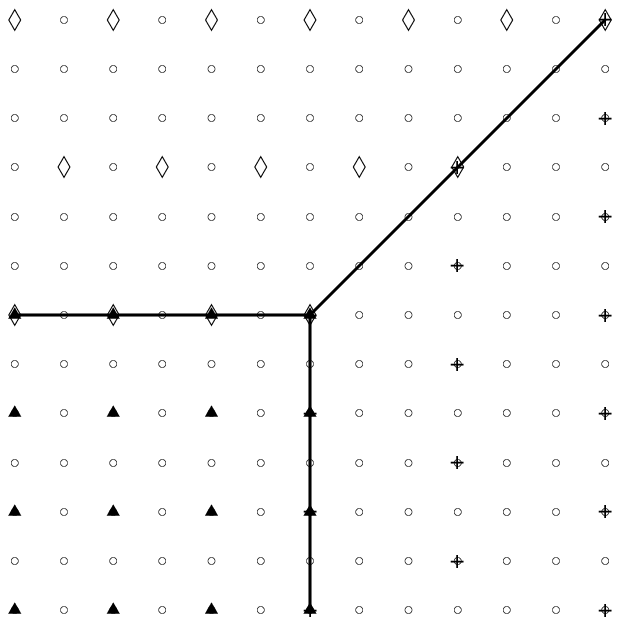}};
\begin{scope}[x={(image3a.east)-(image3a.west)},y={(image3a.north)-(image3a.south)}]
\node[rectangle,fill=white,fill opacity=1] at (0.3,0.3){\small $\rho_{NE}$};
\node[rectangle,fill=white,fill opacity=1] at (-0.475,0){\small $\rho_{W}$};
\node[rectangle,fill=white,fill opacity=1] at (0,-0.45){\small $\rho_{S}$};
\node[rectangle,fill=white,fill opacity=1] at (0.6,0.6){\tiny $w_{NE}=3$};
\node[rectangle,fill=white,fill opacity=1] at (-0.85,0){\tiny $w_{W}=2$};
\node[rectangle,fill=white,fill opacity=1] at (0,-0.8){\tiny $w_{S}=2$};
\node[rectangle,fill=white,fill opacity=1] at (-0.15,0.6){\small $\widetilde{\mathfrak{t}}_{\sigma_N,\mathbb{Z}}$};
\node[rectangle,fill=white,fill opacity=1] at (0.6,-0.15){\small $\widetilde{\mathfrak{t}}_{\sigma_E,\mathbb{Z}}$};
\node[rectangle,fill=white,fill opacity=1] at (-0.65,-0.425){\small $\widetilde{\mathfrak{t}}_{\sigma_{SW},\mathbb{Z}}$};
\end{scope}
\end{tikzpicture}}%
\hspace{2em}
\subfloat[LoF 3b][ \\ Sub-lattices of one- and zero-dimensional cones]{
\label{fig_3b}
\begin{tikzpicture}
\node[anchor=center,inner sep=0] (image3b) at (0,0) {\includegraphics[width=0.45\textwidth]{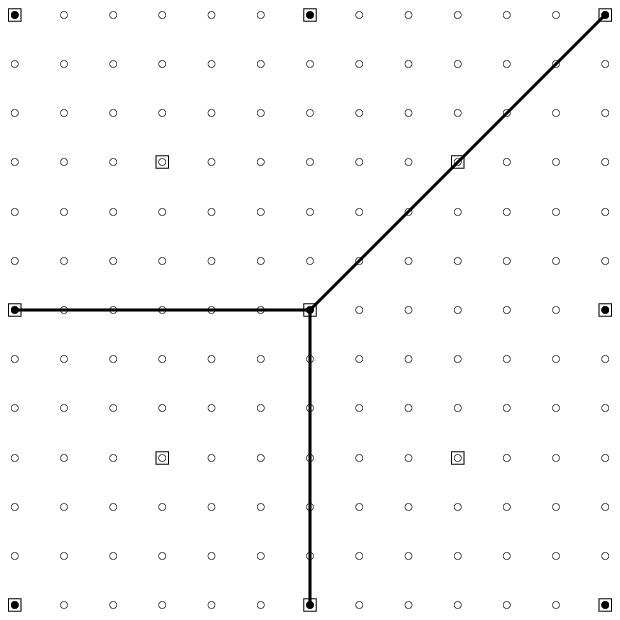}};
\begin{scope}[x={(image3b.east)-(image3b.west)},y={(image3b.north)-(image3b.south)}]
\node[rectangle,fill=white,fill opacity=1] at (0,0.775){\small $\widetilde{\mathfrak{t}}_{\rho_W,\mathbb{Z}} = \widetilde{\mathfrak{t}}_{\rho_S,\mathbb{Z}} = \widetilde{\mathfrak{t}}_{0,\mathbb{Z}}$};
\node[rectangle,fill=white,fill opacity=1] at (0.6,0.3){\small $\widetilde{\mathfrak{t}}_{\rho_{NE},\mathbb{Z}}$};
\end{scope}
\end{tikzpicture}}%
\caption{Fan and fundamental groups of covering tori for $\mathbb{P}^2_{2,2,3}$}
\label{fig_weightedP2}
\end{figure}

The sub-lattice $\mathfrak{t}_{(\Sigma,w),\mathbb{Z}}$ that yields the orbifold fundamental group $\cong \mathbb{Z}/2\mathbb{Z}$ is generated for example by $(2,0)$ and $(1,1)$, and exhibits the orbifold universal cover of this example as isomorphic to the quadric cone (considered in Example \ref{ex_quadric_cone} below).
\end{example}

\section{Picard groups of toric orbifolds}

We proceed to calculate the orbifold Picard group $\Picorb \mathcal{X}_{\Sigma,w}$ and its $\mathbb{T}$-linearized version $\TPicorb \mathcal{X}_{\Sigma,w}$. The toric description of the orbifold structure on $\mathcal{X}_{\Sigma,w}$ we have in place permits to proceed in complete analogy with the underlying projective toric variety $X_\Sigma$.

From the definitions of (linearized) orbi-line bundles in Section \ref{section_orbifolds}, it follows in particular that we have canonically defined subgroups $\Pic^{\widetilde{\mathbb{T}}_\tau} \widetilde{U}_\tau \subset \TPicorb \mathcal{X}_{\Sigma,w}\vert_{U_\tau}$, and it turns out these are sufficient for our purposes.
\begin{theorem}\label{theorem_Picorb} For a toric orbifold $\mathcal{X}_{\Sigma,w}$, the linearized orbifold Picard group admits the following combinatorial descriptions:
\begin{eqnarray*}
 \mathcal{L} \in \TPicorb \mathcal{X}_{\Sigma,w}
 & \leftrightarrow &
 \left\{ m_\sigma \in \widetilde{\mathfrak{t}}^\ast_{\sigma,\mathbb{Z}} \right\}_{\sigma\in\Sigma(n)} \textrm{ s.th. } \forall \sigma_1, \sigma_2 : m_{\sigma_1}-m_{\sigma_2} \in (\sigma_1 \cap \sigma_2 )^\perp \\
 & \leftrightarrow &
 \left\{ h_\rho \in \frac{1}{w_\rho}\rho^\ast_{\mathbb{Z}} \right\}_{\rho\in\Sigma(1)} \ \leftrightarrow \  \left\{ l_\rho = h_\rho(w_\rho\nu_\rho) \in \mathbb{Z} \right\}_{\rho\in\Sigma(1)} .
\end{eqnarray*}
In other words, there is a commutative diagram
\begin{equation}\label{diagram_linearized-orbi-Pic}
 \begin{tikzcd}
  \TPic X_\Sigma \ar[draw=none]{r}[anchor=center]{\subset} \ar[hookrightarrow]{d} &  \TPicorb \mathcal{X}_{\Sigma,w}  \ar{d}{\cong} \\
 \bigoplus_{\rho \in \Sigma(1)} \rho^\ast_{\mathbb{Z}} \ar[draw=none]{r}[anchor=center]{\subset} &
 \bigoplus_{\rho \in \Sigma(1)} \frac{1}{w_\rho}\rho^\ast_{\mathbb{Z}} 
 \end{tikzcd} .
\end{equation}
\end{theorem}

\begin{remark} The fact that the linearized Picard group is described by an \emph{unrestricted} choice of integers associated to the rays of the fan is in complete analogy with line bundles on \emph{non-singular} toric varieties (i.e. toric manifolds), and could be read as a confirmation of the philosophical point of view that ``orbifolds are non-singular''.
\end{remark}

\begin{wrapfigure}{r}{0.33\textwidth}
 \centering
\begin{tikzpicture}
\node[anchor=center,inner sep=0] (image4) at (0,0) {\includegraphics[width=0.33\textwidth,trim=0 0 0 0,clip=]{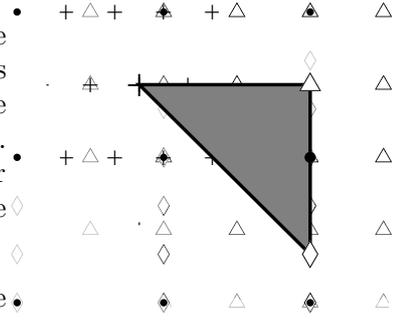}};
\begin{scope}[x={(image4.east)-(image4.west)},y={(image4.north)-(image4.south)}]
\end{scope}
\end{tikzpicture}%
\caption{Classifying data of an orbi-line bundle}\label{fig_olb_weightedP2}
\end{wrapfigure}
\hspace{-2ex}\begin{examplectd}{\ref{ex_P2_2,2,3}}
Just as for toric varieties, line bundles with sections correspond in this way to polytopes in the dual of the Lie algebra; we draw a simple example (of rational Chern class $\frac{7}{6}\mathcal{O}_{\mathbb{P}^2}(1)$ on $\mathbb{P}^2_{2,2,3}$) in Figure \ref{fig_olb_weightedP2}.
Each $m_\sigma$ can be chosen in the corresponding character lattice, subject to the compatibility requirement that the dual fan to the polytope be contained in $\Sigma$.
\end{examplectd}

\begin{proof}[Proof of Theorem A] We proceed as before by using the fact that the restrictions to affine toric subvarieties have no nontrivial automorphisms; therefore any $\mathbb{T}$-linearized orbi-line bundle is still determined by the collection of its restrictions to an open cover
\[
 \mathcal{L} \in \TPicorb \mathcal{X}_{\Sigma,w}
 \leftrightarrow
 \left\{ \mathcal{L}\vert_{U_\tau} \in \TPicorb \mathcal{X}_{\Sigma,w}\vert_{U_\tau} \right\}_{\tau\in\Sigma} .
\]
Now as for any maximal cone $\sigma \in \Sigma(n)$, the orbifold vertex chart $\widetilde{U}_\sigma \to U_\sigma$ represents $\mathcal{X}_{\Sigma,w}\vert_{U_\sigma}$ as finite quotient of an affine space, from the definition we have that
\[
 \TPicorb \mathcal{X}_{\Sigma,w}\vert_{U_\sigma} =
 \Pic^{\widetilde{\mathbb{T}}_\sigma} \widetilde{U}_\sigma =
 \widetilde{\mathfrak{t}}^\ast_{\sigma,\mathbb{Z}},
\]
Any restriction $\mathcal{L}\vert_{U_\tau}$ can be obtained by restricting from some vertex chart, so
\[
 \mathcal{L}\vert_{U_\tau} \in  \Pic^{\widetilde{\mathbb{T}}_\tau} \widetilde{U}_\tau \subset 
 \TPicorb \mathcal{X}_{\Sigma,w}\vert_{U_\tau}
\]
and the collection of characters $m_\tau \in \widetilde{\mathfrak{t}}^\ast_{\tau,\mathbb{Z}}$ has to satisfy the condition
\[
 m_{\tau_1}-m_{\tau_2} \in (\tau_1 \cap \tau_2)^\perp .
\]
If we write the condition for $m_\sigma$ to lie in $\widetilde{\mathfrak{t}}^\ast_{\sigma,\mathbb{Z}}$ in the basis $\{ \nu_{\rho,\sigma}^\ast \}_{\rho \subset \sigma}$, we obtain
\[
 m_\sigma = \sum_{\rho \subset \sigma} h_\rho(\nu_\rho) \nu_{\rho,\sigma}^\ast \quad \textrm{ with } \quad
 h_\rho(\nu_\rho) \in \frac{1}{w_\rho}\mathbb{Z} ,
\]
and the compatibility between the $m_\sigma$ implies that $h_\rho(\nu_\rho)$ is independent of the choice of maximal cone $\sigma \supset \rho$, which concludes the proof of the Theorem.
\end{proof}

\begin{example}\label{ex_quadric_cone}
An example of a two-dimensional toric variety with orbifold singularities (as are all toric singularities in dimension two) is the quadric cone, shown in Figure \ref{fig_quadric-cone}. It defines a non-trivial orbifold structure even for the trivial weight function $w\equiv 1$ we consider here.

\begin{figure}[!htb]
 \centering
\subfloat[LoF 5a][ \\ $X_\Sigma \cong$ \\ {$\{ [W\!:\!X\!:\!Y\!:\!Z] | XZ-Y^2\}$}]{
\label{fig_5a} 
\begin{tikzpicture}
\node[anchor=center,inner sep=0] (image5a) at (0,0) {\includegraphics[width=0.32\textwidth]{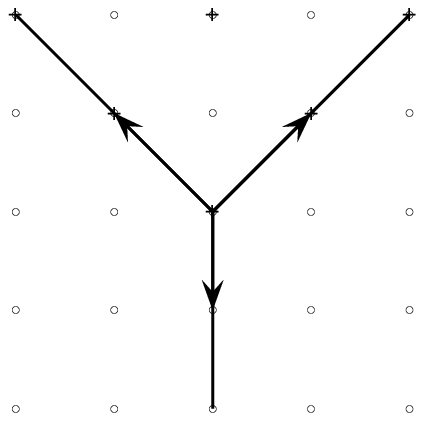}};
\begin{scope}[x={(image5a.east)-(image5a.west)},y={(image5a.north)-(image5a.south)}]
\node[rectangle,fill=white,fill opacity=1] at (0.6,0.6){\small $\nu_{\rho_{NE}}$};
\node[rectangle,fill=white,fill opacity=1] at (-0.6,0.6){\small $\nu_{\rho_{NW}}$};
\node[rectangle,fill=white,fill opacity=1] at (0,-0.6){\small $\nu_{\rho_{S}}$};
\end{scope}
\end{tikzpicture}}%
\subfloat[LoF 5b][ \\ $1\cdot\nu_{\rho_{NE}}^\ast \oplus 1\cdot\nu_{\rho_{NW}}^\ast \oplus 0\cdot\nu_{\rho_{S}}^\ast \mod \mathfrak{t}^\ast_{\mathbb{Z}} = L_{P_1} \in \Pic X_\Sigma$]{
\label{fig_5b} 
\begin{tikzpicture}
\node[anchor=center,inner sep=0] (image5b) at (0,0) {\includegraphics[width=0.32\textwidth]{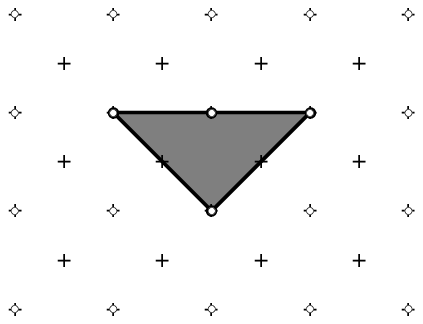}};
\begin{scope}[x={(image5b.east)-(image5b.west)},y={(image5b.north)-(image5b.south)}]
\node[] at (0,0){\small $P_1$};
\end{scope}
\end{tikzpicture}}%
\subfloat[LoF 5c][ \\ $2\cdot\nu_{\rho_{NE}}^\ast \oplus 1\cdot\nu_{\rho_{NW}}^\ast \oplus 0\cdot\nu_{\rho_{S}}^\ast \mod \mathfrak{t}^\ast_{\mathbb{Z}} = \mathcal{L}_{P_2} \in \Picorb \mathcal{X}_{\Sigma,1}$]{
\label{fig_5c} 
\begin{tikzpicture}
\node[anchor=center,inner sep=0] (image5c) at (0,0) {\includegraphics[width=0.32\textwidth]{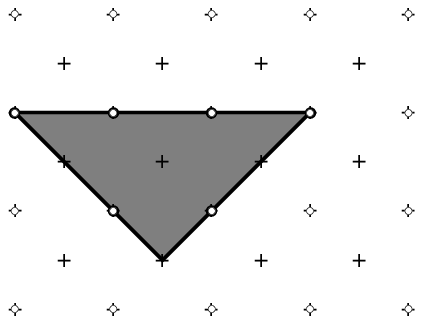}};
\begin{scope}[x={(image5c.east)-(image5c.west)},y={(image5c.north)-(image5c.south)}]
\node[] at (-0.25,-0.2){\small $P_2$};
\end{scope}
\end{tikzpicture}}%
\caption{The quadric cone as orbifold}
\label{fig_quadric-cone}
\end{figure}

The line bundle $L_{P_1}$ defined by the polytope $P_1$ embeds $X_\Sigma$ in $\mathbb{P}^3$ as the quadric cone $XZ-Y^2=0$.

On the same orbifold, consider the second polytope $P_2$ defined by the element  $2\cdot\nu_{\rho_{NE}}^\ast \oplus 1\cdot\nu_{\rho_{NW}}^\ast \oplus 0\cdot\nu_{\rho_{S}}^\ast \mod \mathfrak{t}^\ast_{\mathbb{Z}}$. It corresponds to an orbi-line bundle $\mathcal{L}_{P_2}$ which is not the pull-back of a line bundle, since it is defined by an element outside the image of the left vertical arrow in diagram (\ref{diagram_orbi-Pic}), which is \emph{not} surjective in case $X_\Sigma$ is singular. This shows that considering non-smooth simplicial projective toric varieties as a orbifolds makes a difference regarding the Picard group even with trivial weight function.
\end{example}

\begin{corollary}\label{cor_linearizability}
Any orbi-line bundle admits a lift of the torus action, and hence
\[
 \Picorb \mathcal{X}_{\Sigma,w} = \TPicorb \mathcal{X}_{\Sigma,w} / \mathfrak{t}^\ast_{\mathbb{Z}} .
\]
Thus the following diagram commutes
\begin{equation}\label{diagram_orbi-Pic}
 \begin{tikzcd}
 H^2( X_\Sigma ; \mathbb{Z}) = \Pic X_\Sigma \ar[draw=none]{r}[anchor=center]{\subset} \ar[hookrightarrow]{d} &  \Picorb \mathcal{X}_{\Sigma,w}  \ar{d}{\cong} \\
 \left( \bigoplus_{\rho \in \Sigma(1)} \rho^\ast_{\mathbb{Z}} \right) / \mathfrak{t}^\ast_{\mathbb{Z}}  \ar[draw=none]{r}[anchor=center]{\subset} &
 \left( \bigoplus_{\rho \in \Sigma(1)} \frac{1}{w_\rho}\rho^\ast_{\mathbb{Z}} \right) / \mathfrak{t}^\ast_{\mathbb{Z}}
 \end{tikzcd} ,
\end{equation}
and the torsion subgroup is characterized as either the kernel of the first Chern class or the dual of the orbifold fundamental group,
\[
 \left(\Picorb \mathcal{X}_{\Sigma,w}\right)_{\textrm{tor}} = 
 \ker c_1 =
 \mathfrak{t}^\ast_{(\Sigma,w),\mathbb{Z}} /  \mathfrak{t}^\ast_{\mathbb{Z}} =
 \pi_1^{\textrm{orb}}(\mathcal{X}_{\Sigma,w})\check{\ } .
\]
\end{corollary}

\begin{remark}
Note the analogy with the fact that for topological line bundles on a space $X$
\[
 H^2(X;\mathbb{Z})_{\textrm{tor}} \cong (\pi_1(X)^{\textrm{ab}})_{\textrm{tor}} 
\]
by the universal coefficient and Hurewicz theorems.
\end{remark}

\begin{remark}\label{rmk_integrality-condition}
For non-singular toric varieties the canonical arrow on the left of diagram (\ref{diagram_orbi-Pic}) is of course well known to be an isomorphism
\[
 \Pic X_\Sigma = H^2(X_\Sigma,\mathbb{Z}) = \Big( \bigoplus_{\rho \in \Sigma(1)} \rho^\ast_{\mathbb{Z}} \Big) / \mathfrak{t}^\ast_{\mathbb{Z}}
\]
and is usually expressed using the matrix $N_\Sigma$ made up by the ray generators $\nu_\rho \in \rho\cap\mathfrak{t}_{\mathbb{Z}}$ as
\[
 H^2(X_\Sigma,\mathbb{Z}) \cong (\mathbb{Z}^{\Sigma(1)})^\ast / {}^t\!N_\Sigma \mathfrak{t}^\ast_{\mathbb{Z}} .
\]
Here $\mathbb{Z}^{\Sigma(1)}$ denotes the free $\mathbb{Z}$-module on the set of rays in the fan, the star the dual module, and $N_\Sigma: \mathbb{Z}^{\Sigma(1)} \to \mathfrak{t}_{\mathbb{Z}}$ is the linear map sending the generator associated to a ray $e_\rho$ to the primitive lattice vector in the ray $\nu_\rho$.

Upon tensoring with the reals (or rationals), cohomology classes $c \in H^2(X_\Sigma,\mathbb{R})$ are identified with equivalence classes of vectors $h_c + {}^t\!N_\Sigma \mathfrak{t}^\ast \in (\mathbb{R}^{\Sigma(1)})^\ast / {}^t\!N_\Sigma \mathfrak{t}^\ast$, and the integrality condition takes the form
\[
 c \in H^2(X_\Sigma,\mathbb{Z}) \subset  H^2(X_\Sigma,\mathbb{R})
 \quad \iff \quad
 ( h_c + {}^t\!N_\Sigma \mathfrak{t}^\ast ) \cap (\mathbb{Z}^{\Sigma(1)})^\ast \neq \emptyset ,
\]
integrality of a class $c$ is determined by whether the corrseponding subspace $h_c + {}^t\!N_\Sigma \mathfrak{t}^\ast$ does intersect the lattice of the ambient space.

It is worth noting that for more general $\Sigma$, the subspace ${}^t\!N_\Sigma \mathfrak{t}^\ast_{\mathbb{Z}} \subset (\mathbb{Z}^{\Sigma(1)})^\ast$ and the canonical identification $H^2(X_\Sigma,\mathbb{Q}) \cong (\mathbb{Q}^{\Sigma(1)})^\ast / {}^t\!N_\Sigma \mathfrak{t}^\ast_{\mathbb{Q}}$ are still well-defined, but the identification of line bundles with translates of ${}^t\!N_\Sigma \mathfrak{t}^\ast$ that meet the original lattice breaks down, as we saw in Example \ref{ex_quadric_cone}. It resurrects for orbi-line bundles, however, as we will see shortly.
\end{remark}

\begin{proof}
From Definition \ref{dfn_orbi-line-bundle}, $\mathcal{L}$ is defined by a set of $\widetilde{\Gamma}_\tau$-equivariant line bundles $L_\tau \to \widetilde{U}_\tau$ together with bundle isomorphisms $\phi_{\tau_2,\tau_1}: \psi^{\tau_1}_{\tau_2}{}^\ast L_{\tau_2}\left\vert_{U_{\tau_1}}\right. \to L_{\tau_2}$ for $\tau_1 \subset \tau_2$. As every vertex chart $\widetilde{U}_\sigma$ for a maximal cone is an affine space, all $L_\tau$ are trivial line bundles, and we can choose non-vanishing sections
\[
 1_\tau \in H^0(\widetilde{U}_\tau,L_\tau^\times) .
\]
We may assume these to coincide at some point over the open orbit.
Associated to each non-vanishing section is a character $\widetilde{\chi}_\tau^{\mathcal{L}}: \widetilde{\Gamma}_\tau \to U(1)$ in whose eigenspace $1_\tau$ lies.

Comparing two trivializing sections $1_\tau, 1_{\tau'}$ on $\widetilde{U}_{\tau \cap \tau'}$ we find that they have to differ by a non-vanishing function, that (as $\widetilde{U}_{\tau \cap \tau'}$ is affine $\widetilde{\mathbb{T}}_{\tau \cap \tau'}$-toric) has to be given by a character
\begin{equation}\label{1_sigma_transition}
 \psi^{\tau \cap \tau'}_{\tau}{}^\ast 1_\tau  = \chi_{h_{\tau \tau'}}  \psi^{\tau \cap \tau'}_{\tau'}{}^\ast 1_{\tau'}
 \qquad \textrm{ where } \qquad h_{\tau\tau'}\in \widetilde{\mathfrak{t}}^\ast_{\tau \cap \tau',\mathbb{Z}} .
\end{equation}
Now consider any maximal cone $\sigma_0$ and fix a lift of the torus action $\widetilde{\mathbb{T}}_{\sigma_0}$ to $L_{\sigma_0}$ -- in other words, choose a representative $m_{\sigma_0} \in \widetilde{\mathfrak{t}}^\ast_{\sigma_0,\mathbb{Z}}$ for the character $\widetilde{\chi}_{\sigma_0}^{\mathcal{L}} = m_{\sigma_0} + \mathfrak{t}^\ast_{\mathbb{Z}}$. Define
\[
 m_\sigma := m_{\sigma_0}+h_{\sigma \sigma_0} .
\]
We claim that for all maximal cones $\sigma\in\Sigma(n)$
\begin{eqnarray*}
 & \textrm{(i)} &  m_\sigma + \mathfrak{t}^\ast_{\mathbb{Z}} = \widetilde{\chi}_{\sigma}^{\mathcal{L}} ; \\
 & \textrm{(ii)} & m_\sigma \in \widetilde{\mathfrak{t}}^\ast_{\sigma,\mathbb{Z}} ; \\
 & \textrm{(iii)} & m_\sigma - m_{\sigma'} \in (\sigma \cap \sigma')^\perp ,
\end{eqnarray*}
and note that by Theorem \ref{theorem_Picorb} (or also directly from the definition) this suffices to linearize $\mathcal{L}$. But (i) follows from the fact that each trivializing section $1_\sigma$ lies in the corresponding eigenspace of $\widetilde{\chi}_{\sigma}^{\mathcal{L}}$ and the identity (\ref{1_sigma_transition}), and in turn implies (ii). The last point (iii) holds since (\ref{1_sigma_transition}) holds on all $\widetilde{U}_{\sigma\cap\sigma'}$, whence the character $\chi_{m_\sigma-m_{\sigma'}}$ must not vanish on the minimal orbit in $\widetilde{U}_{\sigma\cap\sigma'}$ whose stabilizer has Lie algebra $\langle \sigma \cap \sigma' \rangle$.
\end{proof}

\begin{corollary}\label{cor_orbi-c1}
For any orbi-line bundle $\mathcal{L}$, there exist integers $d>0,l_\rho$ such that $w_\rho \vert d$ for all rays $\rho \in \Sigma(1)$ and
\[
 \mathcal{L}^d \cong \pi^\ast \mathcal{O}_{X_\Sigma}(\sum_{\rho \in \Sigma(1)} l_\rho \frac{d}{w_\rho}D_\rho) .
\]
Thus the image $c_1(\Picorb \mathcal{X}_{\Sigma,w}) \subset \Pic X_\Sigma \otimes \mathbb{Q}$ is isomorphic to the bottom right entry in the following commutative diagram
\begin{equation}\label{diagram_rational_c1}
 \begin{tikzcd}
 \Picorb \mathcal{X}_{\Sigma,w} \ar{r}{c_1} \ar{d}{\cong} &  \Pic X_\Sigma \otimes \mathbb{Q}  \\
 \left( \bigoplus_{\rho \in \Sigma(1)} \frac{1}{w_\rho}\rho^\ast_{\mathbb{Z}} \right) / \mathfrak{t}^\ast_{\mathbb{Z}}\ar[two heads]{r} & \left( \bigoplus_{\rho \in \Sigma(1)} \frac{1}{w_\rho}\rho^\ast_{\mathbb{Z}} \right) / \mathfrak{t}^\ast_{(\Sigma,w),\mathbb{Z}} \ar[hookrightarrow]{u}
 \end{tikzcd} .
\end{equation}
In particular, a rational divisor class $c \in \Pic X_\Sigma \otimes \mathbb{Q} = H^2(X_\Sigma,\mathbb{Q})$ is the first Chern class of an orbi-line bundle if and only if the corrsponding subspace ${}^t\!N_\Sigma \mathfrak{t}^\ast$ in $(\mathbb{Q}^{\Sigma(1)})^\ast$ meets the new lattice $\bigoplus \frac{1}{w_\rho}\mathbb{Z}e_\rho^\ast$
\[
 ( h_c + {}^t\!N_\Sigma \mathfrak{t}^\ast ) \cap \bigoplus_{\rho \in \Sigma(1)} \frac{1}{w_\rho}\mathbb{Z}e_\rho^\ast  \neq \emptyset .
\]
\end{corollary}

\begin{examplectd}{\ref{ex_quadric_cone}} For any simplicial fan $\Sigma$, such as the one defining the quadric cone, the orbi-line bundles with respect to the canonical orbifold structure $\mathcal{X}_{\Sigma,1}$ are thus characterized by the same condition as line bundles in the smooth case,
\[
 c \in c_1(\Pic \mathcal{X}_{\Sigma,1})
 \subset H^2(X_\Sigma,\mathbb{R})
 \quad \iff \quad
 ( h_c + {}^t\!N_\Sigma \mathfrak{t}^\ast ) \cap (\mathbb{Z}^{\Sigma(1)})^\ast \neq \emptyset .
\]
\end{examplectd}

\begin{examplectd}{\ref{ex_P2_2,2,3}}
Returning to the example of the projective plane with orbifold structures of weights $2,2$ and $3$ along the toric divisors, we depict in Figure \ref{fig_olbs_weightedP2} the dual lattices associated to the maximal cones as well as several examples of the smallest moment polytope representing ``orbi-integral symplectic structures''.

\begin{figure}[!htb]
 \centering
\subfloat[LoF 6a][ \\ Dual cones $\sigma\check{{}_{i}}$ and lattices {$\widetilde{\mathfrak{t}}^\ast_{\sigma_i,\mathbb{Z}}$}]{
\label{fig_6a} 
\begin{tikzpicture}
\node[anchor=center,inner sep=0] (image6a) at (0,0) {\includegraphics[width=0.33\textwidth]{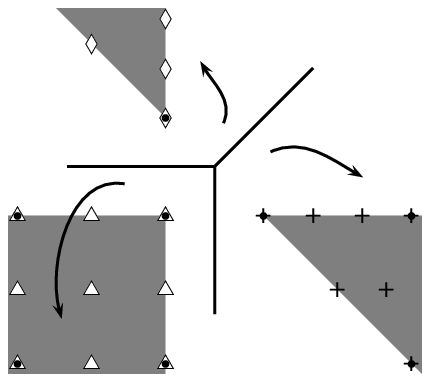}};
\begin{scope}[x={(image6a.east)-(image6a.west)},y={(image6a.north)-(image6a.south)}]
\node[] at (0,0.3){\small $\sigma_N$};
\node[] at (-0.25,0){\small $\sigma_{SW}$};
\node[] at (0.2,0.1){\small $\sigma_E$};
\node[] at (-0.05,0.8){\small $\sigma\check{{}_N}$};
\node[] at (-0.7,-0.8){\small $\sigma\check{{}_{SW}}$};
\node[] at (0.75,0.05){\small $\sigma\check{{}_E}$};
%\node[rectangle,fill=white,fill opacity=1] at (-0.6,0.6){\small $\nu_{\rho_{NW}}$};
%\node[rectangle,fill=white,fill opacity=1] at (0,-0.6){\small $\nu_{\rho_{S}}$};
\end{scope}
\end{tikzpicture}}%
\hspace{2em}
\subfloat[LoF 6b][ \\ Several orbi-line bundles of fractional Chern class]{
\label{fig_6b} 
\begin{tikzpicture}
\node[anchor=center,inner sep=0] (image6b) at (0,0) {\includegraphics[width=0.33\textwidth]{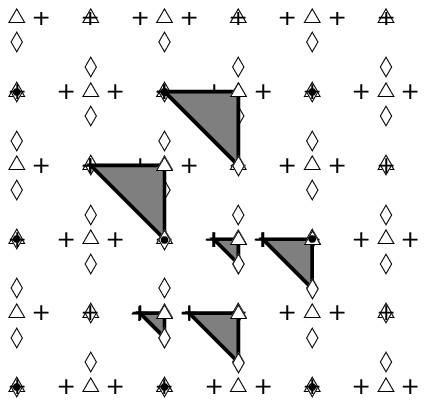}};
\begin{scope}[x={(image6b.east)-(image6b.west)},y={(image6b.north)-(image6b.south)}]
\end{scope}
\end{tikzpicture}}%
\caption{Orbi-line bundles on $\mathbb{P}^2_{2,2,3}$}
\label{fig_olbs_weightedP2}
\end{figure}
It is in particular visible that the rational Chern class alone (given by the side length of the triangle in question) does determine elements in the orbifold Picard group only up to the action of the dual of the orbifold fundamental group.

Using Theorem \ref{theorem_h0} below also the dimension of global spaces is easily read off this diagram: while both orbi-line bundles with $c_1(\mathcal{L}) = \frac{1}{6}\mathcal{O}_{\mathbb{P}^2}(1)$ have no global sections, for $c_1(\mathcal{L}) = \frac{1}{3}\mathcal{O}_{\mathbb{P}^2}(1)$ one has a section while the other has none, and for $c_1(\mathcal{L}) = \frac{1}{2}\mathcal{O}_{\mathbb{P}^2}(1)$ both have one section.
\end{examplectd}

\begin{remark}\label{rmk_Fantechi}
It might be instructive to compare some of these results with those of \cite{fantechi.mann.nironi:2010}: associated to each orbi-line bundle we have a collection of characters, either of the stabilizer groups along the torus orbits
\[
 \chi_\tau^{\mathcal{L}} : \Gamma_\tau \to U(1) ,
\]
or of the covering groups of the developable charts
\[
 \widetilde{\chi}_\tau^{\mathcal{L}} : \widetilde{\Gamma}_\tau \to U(1) .
\]
In their entirety these define an element in the projective limit
\[
 \lim_{\underset{\tau \in \Sigma}{\longleftarrow}} \Gamma\check{{}_\tau}
 \cong
 \lim_{\underset{\tau \in \Sigma}{\longleftarrow}} \widetilde{\Gamma}\check{{}_\tau} .
\]
In case the underlying variety $X_\Sigma$ is smooth (which is the case if and only if the left arrow of diagram (\ref{diagram_orbi-Pic}) is an isomorphism), this homomorphism gives the right half of the short-exact sequence defined by pull-back
\[
 0 \to \Pic X_\Sigma \to \Picorb \mathcal{X}_{\Sigma,w} \to  \lim_{\underset{\tau \in \Sigma}{\longleftarrow}} \Gamma\check{{}_\tau} \to 0,
\]
and the projective limit equals the product (or sum) of the duals of the (cyclic) stabilizer groups $\Gamma_\rho \cong \mathbb{Z}/w_\rho\mathbb{Z}$ of all $n-1$-dimensional torus orbits; in particular this recovers the short-exact sequence (5.6) of \cite{fantechi.mann.nironi:2010}.

In general, the left term of the short-exact sequence \cite[(5.2)]{fantechi.mann.nironi:2010} is the
Picard stack of the associated canonical stack, and by \cite[Rmk. 4.5.(2)]{fantechi.mann.nironi:2010} coincides with the class group $A^1(X_\Sigma)$ of the underlying toric variety, which in our setting is recovered as the orbifold Picard group $\Picorb \mathcal{X}_{\Sigma,1}$.
\end{remark}

\section{Global sections}

To any linearized orbi-line bundle, or compatible collection of characters $m_\sigma \in \widetilde{\mathfrak{t}}^\ast_{\sigma,\mathbb{Z}}$ we can associate a (possibly degenerate or empty) convex \emph{Newton polytope}
\[
 P_{ \{m_\sigma \}_{\sigma \in \Sigma(n)} } :=  \bigcap_{\sigma \in \Sigma(n)} (m_\sigma+\sigma\check{\ }) ;
\]
for ample line bundles on the toric variety, this is the usual Newton polytope. Also as in the case of toric varieties, if the collections $\{m_\sigma\}_{\sigma\in\Sigma(n)}$ and $\{l_\rho\}_{\rho \in \Sigma(1)}$ correspond to each other under the identification in Theorem \ref{theorem_Picorb}, the Newton polytope is alternatively described as intersection of half-spaces
\[
 P = \bigcap_{\rho \in \Sigma(1)} \{x \in \mathfrak{t}^\ast : \langle x, w_\rho\nu_\rho \rangle \geq l_\rho \} .
\]

\begin{theorem}\label{theorem_h0} For any linearized orbi-line bundle $\mathcal{L} \leftrightarrow \{m_\sigma \in \widetilde{\mathfrak{t}}^\ast_{\sigma,\mathbb{Z}} \}_{\sigma\in\Sigma(n)}$ satisfying the previous condition, the dimension of the space of global sections is
\[
 h^0(\mathcal{X}_{\Sigma,w},\mathcal{L}) = \# \mathfrak{t}^\ast_{\mathbb{Z}} \cap P_{ \{m_\sigma \}_{\sigma \in \Sigma(n)} } ;
\]
more precisely, each character in the intersection on the right-hand side appears with multiplicity one in the character decomposition of the representation $\mathbb{T} \circlearrowright H^0(\mathcal{L})$.
\end{theorem}
\begin{proof} This is practically obvious from the construction at this point: over a vertex orbifold chart $\widetilde{U}_\sigma \to U_\sigma$, our line bundle has the linearized trivialization $ \widetilde{\mathbb{T}}_\sigma \circlearrowright \widetilde{U}_\sigma \times \mathbb{C}$ where the action on $\mathbb{C}$ is given by the character $m_\sigma \in \widetilde{\mathfrak{t}}^\ast_{\sigma,\mathbb{Z}}$. Its holomorphic sections are by definition
\[
 H^0(\mathcal{X}_{\Sigma,w}\vert_{U_\sigma},\mathcal{L}\vert_{U_\sigma}) = H^0(\widetilde{U}_\sigma,\widetilde{U}_\sigma\times\mathbb{C})^{\widetilde{\Gamma}_\sigma} ,
\]
and they come with a natural $\mathbb{T}$-action: for any $t\in\mathbb{T}$, we choose an arbitrary lift $\widetilde{t} \in \widetilde{\mathbb{T}}_\sigma$ and set
\[
 (t \cdot \widetilde{s}) (\widetilde{x}) = \widetilde{t} \widetilde{s} (\widetilde{t}^{-1}\widetilde{x}),\qquad \forall \widetilde{x} \in \widetilde{U}_\sigma, \widetilde{s} \in  H^0(\mathcal{X}_{\Sigma,w}\vert_{U_\sigma},\mathcal{L}\vert_{U_\sigma}).
\]
This is well-defined since alternative choices for $\widetilde{t}$ differ by elements of $\widetilde{\Gamma}_\sigma$ that do not affect $\widetilde{s}$ (nor its image). As the character decomposition of the $\widetilde{\mathbb{T}}_{\sigma}$-action on $\widetilde{U}_\sigma \times \mathbb{C}$ is given by
\[
 \widetilde{\mathbb{T}}_{\sigma} \circlearrowright H^0(\widetilde{U}_\sigma,\widetilde{U}_\sigma\times\mathbb{C})
 \cong
 \bigoplus_{\widetilde{m} \in \widetilde{\mathfrak{t}}^\ast_{\sigma,\mathbb{Z}} \cap (m_\sigma+\sigma\check{\ })}
 \mathbb{C} \widetilde{s}_{\chi_{\widetilde{m}}} ,
\]
the $\widetilde{\Gamma}_\sigma$-invariant subspace together with its $\mathbb{T}$-weights is obtained by restricting to the characters that pull back from $\mathbb{T}$,
\[
 \mathbb{T} \circlearrowright H^0(\widetilde{U}_\sigma,\widetilde{U}_\sigma\times\mathbb{C})^{\widetilde{\Gamma}_\sigma}
 \cong
 \bigoplus_{m \in \mathfrak{t}^\ast_{\mathbb{Z}} \cap (m_\sigma+\sigma\check{\ })}
 \mathbb{C} \widetilde{s}_{\chi_m} .
\]
By patching these local conditions over all vertex charts, and hence over all of $\mathcal{X}_{\Sigma,w}$ the claimed description of the space of global sections follows.
\end{proof}

\begin{remark}\label{rmk_Sakai}
Our description of global sections at first sight seems to be in conflict with the main result of the preprint \cite{sakai:2014}, where it is claimed (Theorem 4.4) that non-integral polytopes on stacks that would correspond to non-simply connect orbifolds have $0$ as quantization: in fact it seems to us that Assumption 4.2. of that paper implies that the second clause in Theorem 4.4 does not occur, and hence there is no conflict with our results.

The ``genuine'' orbi-line bundles \emph{not} in $\pi^\ast \Pic \mathbb{P}^1 \subset \Picorb \mathbb{P}^1_{n,n}$ in \cite[Examples 4.3.(iv) and (v)]{sakai:2014} in fact do not arise from the reduction process employed there, while they are accessible by the glueing construction we use.
\end{remark}

\section{Applications}

In this section we present two (related) applications for the combinatorial description we have achieved: first, we discuss which Chern classes on toric orbifolds are represented by orbi-line bundles, and then we apply these results to determine explicitly the ``Bohr--Sommerfeld orbi-line bundles'' associated to the restriction of a toric action to a subtorus.

\subsection{Quantization of symplectic toric orbifolds}\label{sect_application}

As an application of our results we consider the classification of symplectic toric orbifolds by Lerman and Tolman and determine which of the symplectic forms $\omega_{P,w}$ are represented by orbi-line bundles. 

\begin{theorem*}[\cite{lerman.tolman:1997}] Symplectic toric orbifolds $(\mathcal{X},\omega)$ are in bijective correspondence with convex polytopes $P \subset \mathfrak{t}^\ast$ up to translation, and weight functions $w$ on the set of facets such that
\begin{itemize}
 \item the dual fan $\Sigma_P$ of $P$ is defined in $\mathfrak{t}^\ast_{\mathbb{Z}} \otimes \mathbb{Q}$ and simplicial, and
 \item facets $F$ of $P$ (or rays of $\Sigma_P$) are decorated with positive integers $w_F \in \mathbb{N}_+$ .
\end{itemize}
\end{theorem*}

We denote the corresponding symplectic toric orbifold by $(\mathcal{X}_{\Sigma_P,w},\omega_{P,w})$, and note \cite[Thm. 1.7]{lerman.tolman:1997} that it carries a canonical complex analytic (or algebraic) structure.
Just as in the case of toric varieties the actual identification of the symplectic and the complex construction is \emph{not} unique, but its choice is irrelevant to answering the following questions:

\begin{question} When does the symplectic form $\omega_{P,w}$ in this classification represent the first Chern class of an orbi-line-bundle $\mathcal{L} \in \Picorb \mathcal{X}_{\Sigma_P,w}$?
\end{question}
\begin{question} If this is the case, and given a compatible complex structure on $\mathcal{X}_{\Sigma_P,w}$, what is the dimension of the space of sections $H^0(\mathcal{X}_{\Sigma_P,w},\mathcal{L})$?
\end{question}

\begin{theorem}\label{theorem_prequantizability} Let $P \subset \mathfrak{t}^\ast$ be a polytope with rational simplicial dual fan $\Sigma=\Sigma_P$ and weight function $w$, and consider the toric orbifold $\mathcal{X}_{\Sigma,w} \to X_{\Sigma}$; then the following are equivalent:
\begin{itemize}
 \item[(i)] there exists an orbi-line bundle $\mathcal{L}\to\mathcal{X}_{\Sigma,w}$ representing the class defined by $P$ in $\Pic X_\Sigma \otimes \mathbb{R}$ (which therefore is a fortiori rational)
\[
 c_1(\mathcal{L}) = [\omega_{P,w}] \in \Pic X_\Sigma \otimes \mathbb{Q} \cong H^2(X_\Sigma,\mathbb{Q}) ;
\]
 \item[(ii)] there exists a translate of $P$ and characters $m_\sigma \in \widetilde{\mathfrak{t}}^\ast_{\sigma,\mathbb{Z}}$ such that
\[
 P = \bigcap_{\sigma \in \Sigma(n)} (m_\sigma+\sigma\check{\ }) =
 \conv \left\{ m_\sigma \right\}_{\sigma \in \Sigma(n)} ;
\]
\item[(iii)]
there exist integers $l_\rho \in \mathbb{Z}$ such that $P$ is a translate of the intersection
\[
 P = \bigcap_{\rho \in \Sigma(1)} \{x \in \mathfrak{t}^\ast : \langle x, w_\rho\nu_\rho \rangle \geq l_\rho \} ;
\]
\item[(iv)] 
the affine subspace defined by the Chern class $[\omega_{P,w}]$ intersects the weighted character lattice nontrivially,
\[
 ( h_{[\omega_{P,w}]} + {}^t\!N_\Sigma \mathfrak{t}^\ast ) \cap \bigoplus_{\rho \in \Sigma(1)} \frac{1}{w_\rho}\mathbb{Z}e_\rho^\ast \neq \emptyset .
\]
\end{itemize}
In this case, the non-equivalent orbi-line bundles representing this rational Chern class are a torsor under the group of characters of the orbifold fundamental group.
\end{theorem}

\begin{proof}
The equivalence of (i) and (ii) is immediate from our characterization of the orbifold Picard group in Theorem \ref{theorem_Picorb} and Satake's isomorphism between the orbifold de Rham cohomology (cf. \cite{satake:1956} or \cite[Thm. (1.9)]{blache:1996}), which fits into the diagram
\[
 \begin{tikzcd}
 \Picorb \mathcal{X}_{\Sigma,w} \ar{r} \ar{d} & \Pic X_\Sigma \otimes \mathbb{R} \ar{d}{\cong} \\
 H^2_{dR}(\mathcal{X}_{\Sigma,w},\mathbb{R}) \ar{r}{\textrm{Satake}}[swap]{\cong} & H^2(X_\Sigma,\mathbb{R})
 \end{tikzcd} .
\]
The kernel of the top (and thus by the diagram also of the left) morphism is the torsion subgroup, which by Corollary \ref{cor_linearizability} coincides with the dual of the orbifold fundamental group.

As for (iii), it is tantamount to (ii) by the alternative descriptions of Newton polytopes as either intersections of dual maximal cones or half-spaces, and (iv) is the characterization of Chern classes of orbi-line bundles from Corollary \ref{cor_orbi-c1}.

Finally, the last assertion is just a restatement of Corollary \ref{cor_linearizability}.
\end{proof}

Directly from Theorem \ref{theorem_h0} now follows the answer to the second question:
\begin{corollary}
  For $\mathcal{L} = \mathcal{L}_{\{m_\sigma\}}$ with $m_\sigma \in \widetilde{\mathfrak{t}}^\ast_{\sigma,\mathbb{Z}}$  and $P = \conv \{ m_\sigma \in \widetilde{\mathfrak{t}}^\ast_{\sigma,\mathbb{Z}} \}$ as in Theorem \ref{theorem_prequantizability} (ii), the dimension of the quantization space $H^0(\mathcal{X}_{\Sigma_P,w},\mathcal{L})$ is given by the number of lattice points in the Newton polytope
\[
 h^0(\mathcal{X}_{\Sigma_P,w},\mathcal{L}) = \# P \cap \mathfrak{t}^\ast_{\mathbb{Z}} .
\]
\end{corollary}

\subsection{Reduction with respect to a sub-torus and Bohr--Sommerfeld conditions}

While this is not the place to enter the details of geometric quantization in mixed polarizations (that we hope to come back to elsewhere), we here wish to illustrate the relevance of our main results for this subject in the toric case.

Consider a toric orbifold $\mathcal{X}_{P,w}$ and a sub-torus $\mathbb{T}_1 \subset \mathbb{T}$, so that we have a diagram
\[
  \begin{tikzcd}
    \mathbb{T}_1 \arrow[r,phantom,"\subset"] & \mathbb{T} \arrow[r,phantom,"\circlearrowright"] \arrow[d] & (\mathcal{X}_{P,w},\omega_{P,w}) \arrow[r, "\mu"] \arrow[rd, "\mu_1"] \arrow[d, swap, "p"] & P \arrow[r,phantom, "\subset"] \arrow[d, "\pi"] & \mathfrak{t}^\ast \arrow[d,"\pi"] \\
    & \mathbb{T} / \mathbb{T}_1 \arrow[r,phantom,"\circlearrowright"] & \mathcal{X}_{P,w} / \mathbb{T}_1 \arrow[r, swap, "q"] & P_1 \arrow[r,phantom, "\subset"] & \mathfrak{t}^\ast_1 \\
  \end{tikzcd} .
\]
\begin{proposition} For any regular value $\alpha \in P_1$, the symplectic reduction
  \[
    \mathcal{X}_{\alpha} := \mu_1^{-1}(\alpha) / \mathbb{T}_1
  \]
is a $\mathbb{T}/\mathbb{T}_1$-toric orbifold associated with the polytope $P_\alpha = \pi^{-1}(\alpha) \cap P$, 
\[
 \mathcal{X}_\alpha \cong \mathcal{X}_{P_\alpha,w^{\alpha}} ,
\]
where the weights are determined by the following condition: for any facet $F_\alpha \subset P_\alpha$ and $F$ the corresponding facet of $P$ such that $F_\alpha = F\cap P_\alpha$,
\[
 w_{F}\nu_{F} + (\mathfrak{t}_1)_{\mathbb{Z}} = w_{F_{\alpha}}^\alpha \nu_{F_{\alpha}} .
\]
\end{proposition}
\begin{proof}
  The characterization of the moment polytope $P_\alpha$ for the $\mathbb{T}/\mathbb{T}_1$-action on $\mathcal{X}_\alpha$ follows from \cite[Lemma 3.9 and Theorem 5.2]{lerman.tolman:1997}.

  Considering the circle $\mathbb{S}^1_F \subset \mathbb{T}$ that stabilizes points in the relative interior of facet $F\subset P$, it surjects onto the corresponding stabilizer
  \[
    \mathbb{S}^1_F \twoheadrightarrow \mathbb{S}^1_{F_\alpha} \subset \mathbb{T}/\mathbb{T}_1
  \]
  with kernel $\mathbb{S}^1_F \cap \mathbb{T}_1 \cong \mathbb{Z}/ k \mathbb{Z}$, where $k$ is determined by
\[
 \nu_{F} + (\mathfrak{t}_1)_{\mathbb{Z}} = k \nu_{F_{\alpha}} ,
\]
  and the condition for the weights follows.
\end{proof}

\begin{proposition} The following conditions are equivalent:
  \begin{itemize}
    \item[(i)] the restriction of any orbi-line bundle $\mathcal{L}$ with $c_1(\mathcal{L}) = [\omega_{P,w}]$ to a fiber $\mu_1^{-1}(\alpha)$ descends to an orbi-line bundle on the symplectic reduction $\mathcal{X}_\alpha$,
  \[
  \exists \mathcal{L}_\alpha \in \Picorb{\mathcal{X}_\alpha} \textrm{ s.th. }
  \begin{tikzcd}
    \left.\mathcal{L}\right\vert_{\mu_1^{-1}(\alpha)} \cong p^\ast \mathcal{L}_\alpha \arrow[r] \arrow[d] & \mathcal{L}_\alpha \arrow[d] \\
    \mu_1^{-1}(\alpha) \arrow[r, "p"] & \mathcal{X}_\alpha
  \end{tikzcd} ;
  \]
\item[(ii)] the class of the reduced symplectic form is orbi-integral,
  \[
    [\omega_\alpha] \in c_1 \left ( \Picorb(\mathcal{X}_{\alpha}) \right) ;
  \]
    \item[(iii)] the affine subspace $\pi^{-1}(\alpha)$ meets the lattice $\mathfrak{t}^\ast_{\mathbb{Z}}$,
      \[
        \pi^{-1}(\alpha) \cap \mathfrak{t}^\ast_{\mathbb{Z}} \neq \emptyset .
      \]
  \end{itemize}
\end{proposition}
\begin{proof}
  By Theorem \ref{theorem_prequantizability} we can choose a translate of $P$ such that
  \[
    P = \conv\{m_\sigma\},
    \qquad \text{ and } \qquad \mathcal{L} \cong \mathcal{L}_{\{m_\sigma\}} .
  \]
For the restriction of $\mathcal{L}$ to $\mu_1^{-1}(\alpha)$ to descend to an orbi-line bundle, all stabilizers (in $\mathbb{T}_1$) at points $m\in\mu_1^{-1}(\alpha)$ have to act trivially on the fiber of $\mathcal{L}$ at $m$, which is achieved under condition (iii) by multiplying the $\mathbb{T}_1$-action by the character $\chi_\alpha^{-1}$. We already had noted that (ii) and (iii) are equivalent in Remark \ref{rmk_integrality-condition}.
\end{proof}
The subvarieties $\mu^{-1}(\alpha)$ satisfying the condition of the proposition are called \emph{Bohr--Sommerfeld leaves}; occasionally we may also term the corresponding $\alpha$ ``Bohr--Som\-mer\-feld'', or ``BS'' for short. It should be noted that they are determined by the rational Chern class of $[\omega_{P,w}]$ alone, independently of the ``pre-quantization''.
  
\begin{remark}\label{rmk_BS-linebundle} The line bundle $\mathcal{L}_\alpha$ in the statement is generally \emph{not} unique up to isomorphism, as clearly any torsion element in the orbifold Picard group has to pull back to the trivial line bundle on the manifold $\mu^{-1}(\alpha)$. This fact is quite relevant to the ``quantization commutes with reduction'' problem, as we have seen before that changing orbi-line bundles by torsion can change the dimension of their space of sections. In greater generality this issue is settled by including the information contained in the connection, but in our toric setting there is a unique choice determined by the requirement of $\mathbb{T}$-equivariance, as the proof shows.
\end{remark}

Summing up, we have almost finished the proof of our final theorem:
\begin{theorem}\label{thm_qr-rq}
  Suppose given a symplectic toric orbifold $\mathcal{X}_{P,w}$ and consider the restriction of the torus action to a subtorus $\mathbb{T}_1 \subset \mathbb{T}$. Then for regular values $\alpha \in P_1 = \mu_1 P$ of the moment map $\mu_1 : \mathcal{X}_{P,w} \to \mathfrak{t}_1^\ast$ of the $\mathbb{T}_1$-action,
  \begin{itemize}
  \item[(i)] the reduced orbifold $\mathcal{X}_\alpha := \mu_1^{-1}(\alpha)/\mathbb{T}_1$ is determined (as a symplectic $\mathbb{T}/\mathbb{T}_1$-toric orbifold) by the intersection of the moment polytope $P$ with the fiber of the canonical projection $\pi: \mathfrak{t}^\ast \to \mathfrak{t}_1^\ast$ over $\alpha$,
    \[
     P_\alpha := P \cap \pi^{-1}(\alpha) ,
    \]
    and weights $w^\alpha$ determined by the condition that for any facet $F_\alpha \subset P_\alpha$ and corresponding facet $F\subset P$ such that $F_\alpha = F\cap P_\alpha$
    \[
     w_F \nu_F + \left( \mathfrak{t}_1 \right)_{\mathbb{Z}} = w_{F_\alpha}^\alpha \nu_{F_\alpha} .
    \]
  \item[(ii)] The reduced symplectic form on $\mathcal{X}_\alpha$ is represented by an orbi-line bundle if and only if
    \[
     \pi^{-1}(\alpha) \cap \mathfrak{t}_{\mathbb{Z}}^\ast \neq \emptyset .
    \]
    These values of $\alpha$ and symplectic reductions are called ``Bohr--Sommerfeld''.
  \item[(iii)] In this case, for any choice of orbi-line bundle $\mathcal{L}_{\{m_\sigma\}} \to \mathcal{X}_{P,w}$ representing the class of the symplectic form $\omega_{P,w}$, there is a unique orbi-line bundle $\mathcal{L}_\alpha$ on the Bohr--Sommerfeld fiber $\mathcal{X}_\alpha$ descending in a $\mathbb{T}$-equivariant manner from $\mu_1^{-1}(\alpha)$.
  \item[(iv)] With these notations, there is a decomposition
  \[
  H^0(\mathcal{X}_{P,w},\mathcal{L}_{\{m_\sigma\}}) \cong
  \bigoplus_{\stackrel{\alpha \in P_1}{Bohr-Sommerfeld}}
  H^0(\mathcal{X}_\alpha,\mathcal{L}_\alpha) .
  \]
\end{itemize}
\end{theorem}
\begin{proof}
 Only statement (iv) remains to be proved, which however follows immediately from the previous Proposition and Theorem \ref{theorem_h0}, as both vector spaces have bases indexed by all integer points in $P$. We furthermore see that the right-hand side corresponds to the isotypical decomposition of the restriction of the toric action to $\mathbb{T}_1$.
\end{proof}

\begin{example}
  Consider the restriction of the standard action $\mathbb{T}^3 \circlearrowright \mathbb{P}^3$, polarized by $\mathcal{O}_{\mathbb{P}^3}(3)$, to the circle generated by the vector $(4,3,6) \in \mathbb{R}^3 \cong \mathfrak{t}$.
  
\begin{figure}[!htb]
 \centering
\subfloat[LoF 7a][ \\ Moment map of restricted action]{
\label{fig_7a} 
\begin{tikzpicture}
\node[anchor=center,inner sep=0] (image7a) at (0,0) {\includegraphics[width=0.33\textwidth]{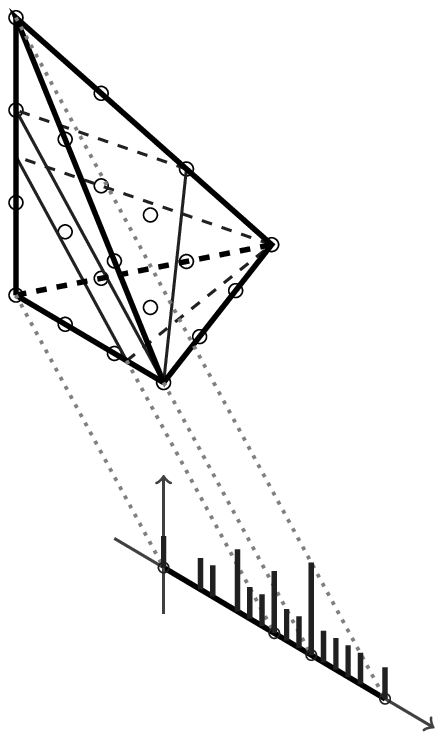}};
\begin{scope}[x={(image7a.east)-(image7a.west)},y={(image7a.north)-(image7a.south)}]
  \node[] at (-1.25,0.5){$\mathfrak{t}^\ast \supset P$};
  \node[] at (0.15,-0.85){$\mathfrak{t}^\ast_1 \supset P_1$};
  \node[] at (-0.55,-0.35){$h^0(\mathcal{L}_\alpha)$};
%\node[rectangle,fill=white,fill opacity=1] at (-0.6,0.6){\small $\nu_{\rho_{NW}}$};
%\node[rectangle,fill=white,fill opacity=1] at (0,-0.6){\small $\nu_{\rho_{S}}$};
\end{scope}
\end{tikzpicture}}%
\hspace{2em}
\subfloat[LoF 7b][\\Weighted fans of reductions]{% wrt\\$e_1 = (-3,2,1), e_2 = (-3,0,2)$]{
\label{fig_7b} 
\begin{tikzpicture}
\node[anchor=center,inner sep=0] (image7b) at (0,0) {\includegraphics[width=0.166\textwidth]{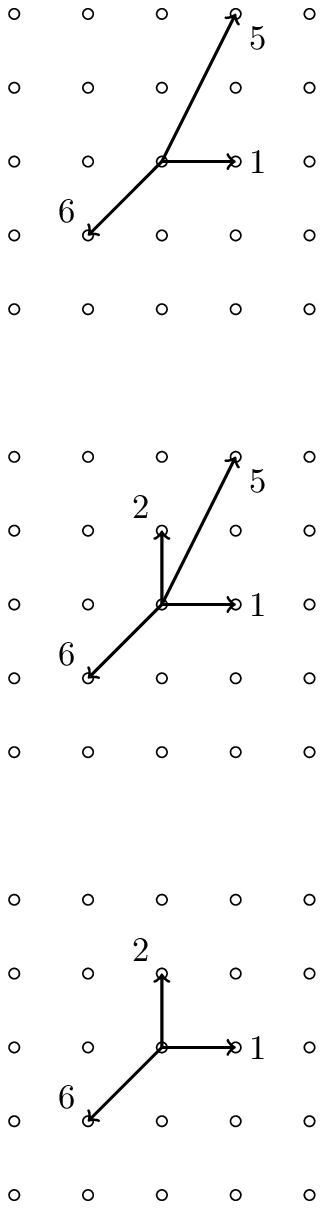}};
\begin{scope}[x={(image7b.east)+15-(image7b.west)},y={(image7b.north)-(image7b.south)}]
  \node[] at (2,0.75){$0 < \mu_1 \leq 9$};
  \node[] at (2,0){$9 < \mu_1 < 12$};
  \node[] at (2,-0.75){$12 \leq \mu_1 < 18$};
\end{scope}
\end{tikzpicture}}%
\caption{Restricting the action on $\mathbb{P}^3$ to a circle}
\label{fig_reduction_P3}
\end{figure}

The moment map $\mu_1$ of the restricted action has image $P_1 = [0,18] \subset \mathbb{R} \cong \mathfrak{t}^\ast_1$ and critical values $0,9,12$ and $18$. The ``critical slices'' of the moment map and the weighted fans of the reductions on each component of regular values are shown in Figure \ref{fig_reduction_P3}.

\begin{figure}[!htb]
 \centering
\begin{tikzpicture}
\node[anchor=center,inner sep=0] (image8) at (0,0) {\includegraphics[width=0.8\textwidth]{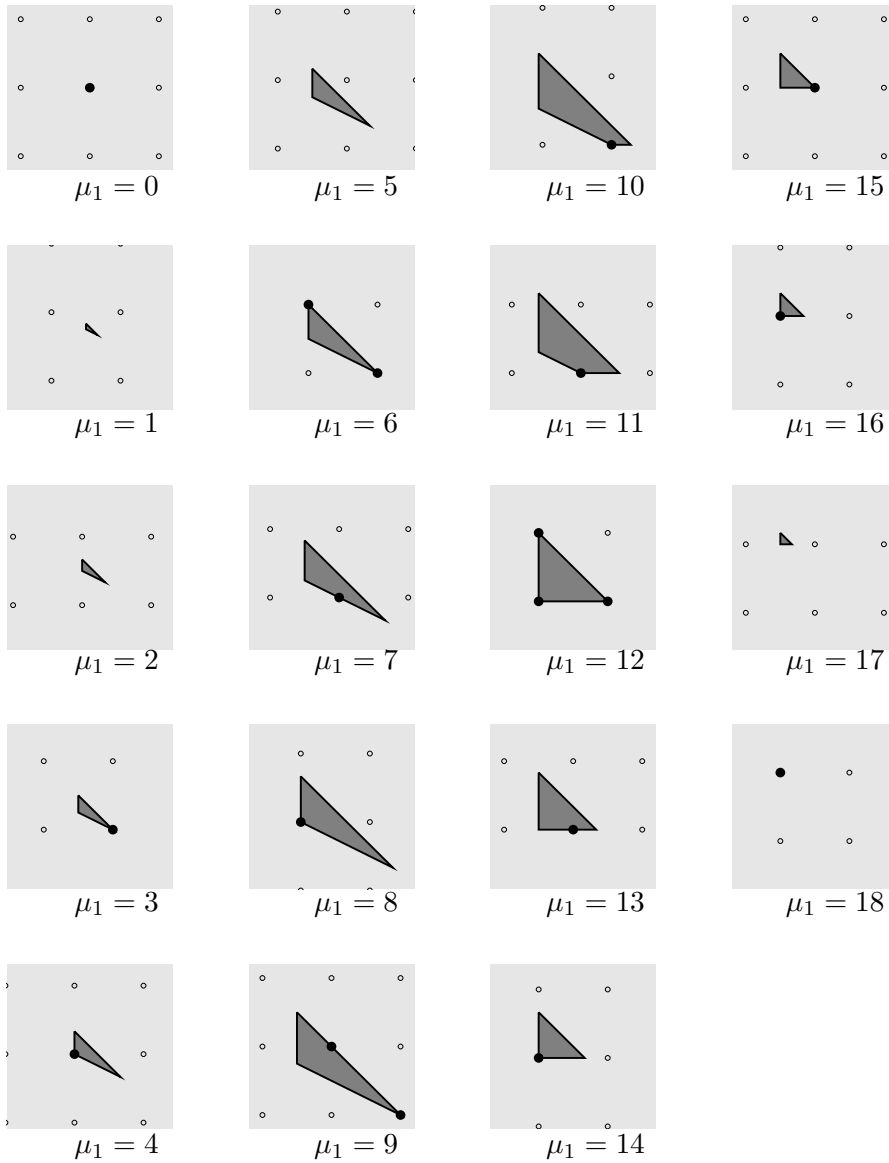}};
\begin{scope}[x={(image8.east)-(image8.west)},y={(image8.north)-(image8.south)}]
  \node[] at (-0.75,0.67){$\mu_1 = 0$};
  \node[] at (-0.75,0.25){$\mu_1 = 1$};
  \node[] at (-0.75,-0.17){$\mu_1 = 2$};
  \node[] at (-0.75,-0.6){$\mu_1 = 3$};
  \node[] at (-0.75,-1.03){$\mu_1 = 4$};
  \node[] at (-0.215,0.67){$\mu_1 = 5$};
  \node[] at (-0.215,0.25){$\mu_1 = 6$};
  \node[] at (-0.215,-0.17){$\mu_1 = 7$};
  \node[] at (-0.215,-0.6){$\mu_1 = 8$};
  \node[] at (-0.215,-1.03){$\mu_1 = 9$};
  \node[] at (0.32,0.67){$\mu_1 = 10$};
  \node[] at (0.32,0.25){$\mu_1 = 11$};
  \node[] at (0.32,-0.17){$\mu_1 = 12$};
  \node[] at (0.32,-0.6){$\mu_1 = 13$};
  \node[] at (0.32,-1.03){$\mu_1 = 14$};
  \node[] at (0.855,0.67){$\mu_1 = 15$};
  \node[] at (0.855,0.25){$\mu_1 = 16$};
  \node[] at (0.855,-0.17){$\mu_1 = 17$};
  \node[] at (0.855,-0.6){$\mu_1 = 18$};
\end{scope}
\end{tikzpicture}%
\caption{Polytopes of the Bohr--Sommerfeld orbi-line bundles (weights ommited, cf. the fans in Figure \ref{fig_7b})}\label{fig_reduced_polytopes}
\end{figure}

In this case the symplectic reductions are orbifolds even for the non-regular values $9$ and $12$, as all two-dimensional toric singularities are orbifold singularities.

The Bohr--Sommerfeld leaves are precisely the pre-images of the integers in $P_1$, and we show the moment polytopes of the symplectic reduction corresponding to the orbi-line bundles of Remark \ref{rmk_BS-linebundle} in Figure \ref{fig_reduced_polytopes}. In particular, several of the Bohr--Sommerfeld fibers with ``small'' Chern classes carry orbi-line bundles without holomorphic sections (namely those corresponding to $\mu_1 = 1,2,5,17$).
\end{example}

We hope to come back to quantizations in mixed polarizations in greater generality, and the aspects of restrictions of Hamiltonian torus actions linked to metric degenerations of the total space in particular in future work.

\section*{Acknowledgements}

The authors would like to thank Rosa Sena--Dias for discussions; this work was partially supported by FCT/Portugal through the projects UID/\allowbreak{}MAT/\allowbreak{}04459/\allowbreak{}2013 and PTDC\allowbreak{}/MAT-\allowbreak{}GEO/\allowbreak{}3319/\allowbreak{}2014.

%%%
%%%\printbibliography
%%%

\bibliographystyle{halpha-abbrv}
\bibliography{toric}

\end{document}